\documentclass{amsart}
\usepackage{amssymb, amsbsy, amsthm, amsmath, amstext, amsopn, verbatim}
\usepackage[all]{xy}
\usepackage{amsfonts}
\usepackage{amscd}
\usepackage{mathrsfs}
\usepackage{verbatim}
\usepackage{color}

\newtheorem{thm}{Theorem} [section]
\newtheorem{lemma}[thm]{Lemma}

\newtheorem{corollary}[thm]{Corollary}
\newtheorem{prop}[thm]{Proposition}

\newtheorem*{rough-thm-1}{Rough Version of Vanishing Theorem}
\newtheorem*{rough-thm-2}{Rough Version of Exactness Theorem}

\theoremstyle{definition}

\newtheorem*{basic convention}{Basic Conventions}

\newtheorem{defn}[thm]{Definition}

\newtheorem{example}[thm]{Example}

\newtheorem{convention}[thm]{Convention}

\theoremstyle{remark}

\newtheorem{remark}[thm]{Remark}
                         
\newtheorem{claim}[thm]{Claim}

\begin{document}

\numberwithin{equation}{section}

\newcommand{\hs}{\mbox{\hspace{.4em}}}
\newcommand{\ds}{\displaystyle}
\newcommand{\bd}{\begin{displaymath}}
\newcommand{\ed}{\end{displaymath}}
\newcommand{\bcd}{\begin{CD}}
\newcommand{\ecd}{\end{CD}}

\newcommand{\proj}{\operatorname{Proj}}
\newcommand{\bproj}{\underline{\operatorname{Proj}}}
\newcommand{\spec}{\operatorname{Spec}}
\newcommand{\bspec}{\underline{\operatorname{Spec}}}
\newcommand{\pline}{{\mathbf P} ^1}
\newcommand{\pplane}{{\mathbf P}^2}
\newcommand{\coker}{{\operatorname{coker}}}
\newcommand{\ldb}{[[}
\newcommand{\rdb}{]]}

\newcommand{\Sym}{\operatorname{Sym}^{\bullet}}
\newcommand{\Symp}{\operatorname{Sym}}
\newcommand{\Pic}{\operatorname{Pic}}
\newcommand{\AAut}{\operatorname{Aut}}
\newcommand{\PAut}{\operatorname{PAut}}

\newcommand{\too}{\twoheadrightarrow}
\newcommand{\C}{{\mathbf C}}
\newcommand{\cA}{{\mathcal A}}
\newcommand{\cS}{{\mathcal S}}
\newcommand{\cV}{{\mathcal V}}
\newcommand{\cM}{{\mathcal M}}
\newcommand{\bA}{{\mathbf A}}
\newcommand{\aline}{\mathbb{A}^1}
\newcommand{\cB}{{\mathcal B}}
\newcommand{\cC}{{\mathcal C}}
\newcommand{\cD}{{\mathcal D}}
\newcommand{\D}{{\mathcal D}}
\newcommand{\cs}{{\mathbf C} ^*}
\newcommand{\boldc}{{\mathbf C}}
\newcommand{\cE}{{\mathcal E}}
\newcommand{\cF}{{\mathcal F}}
\newcommand{\cG}{{\mathcal G}}
\newcommand{\G}{{\mathbf G}}
\newcommand{\fg}{{\mathfrak g}}
\newcommand{\ft}{\mathfrak t}
\newcommand{\bH}{{\mathbf H}}
\newcommand{\cH}{{\mathcal H}}
\newcommand{\cI}{{\mathcal I}}
\newcommand{\cJ}{{\mathcal J}}
\newcommand{\cK}{{\mathcal K}}
\newcommand{\cL}{{\mathcal L}}
\newcommand{\baL}{{\overline{\mathcal L}}}
\newcommand{\M}{{\mathcal M}}
\newcommand{\bM}{{\mathbf M}}
\newcommand{\bm}{{\mathbf m}}
\newcommand{\cN}{{\mathcal N}}
\newcommand{\theo}{\mathcal{O}}
\newcommand{\cP}{{\mathcal P}}
\newcommand{\cR}{{\mathcal R}}
\newcommand{\boldp}{{\mathbf P}}
\newcommand{\boldq}{{\mathbf Q}}
\newcommand{\bbL}{{\mathbf L}}
\newcommand{\cQ}{{\mathcal Q}}
\newcommand{\cO}{{\mathcal O}}
\newcommand{\Oo}{{\mathcal O}}
\newcommand{\OX}{{\Oo_X}}
\newcommand{\OY}{{\Oo_Y}}
\newcommand{\otY}{{\underset{\OY}{\ot}}}
\newcommand{\otX}{{\underset{\OX}{\ot}}}
\newcommand{\cU}{{\mathcal U}}
\newcommand{\cX}{{\mathcal X}}
\newcommand{\cW}{{\mathcal W}}
\newcommand{\boldz}{{\mathbf Z}}
\newcommand{\cZ}{{\mathcal Z}}
\newcommand{\qgr}{\operatorname{qgr}}
\newcommand{\gr}{\operatorname{gr}}
\newcommand{\coh}{\operatorname{coh}}
\newcommand{\End}{\operatorname{End}}
\newcommand{\Hom}{\operatorname{Hom}}
\newcommand{\uHom}{\underline{\operatorname{Hom}}}
\newcommand{\uHomY}{\uHom_{\OY}}
\newcommand{\uHomX}{\uHom_{\OX}}
\newcommand{\Ext}{\operatorname{Ext}}
\newcommand{\bExt}{\operatorname{\bf{Ext}}}
\newcommand{\Tor}{\operatorname{Tor}}

\newcommand{\inv}{^{-1}}
\newcommand{\airtilde}{\widetilde{\hspace{.5em}}}
\newcommand{\airhat}{\widehat{\hspace{.5em}}}
\newcommand{\nt}{^{\circ}}
\newcommand{\del}{\partial}

\newcommand{\supp}{\operatorname{supp}}
\newcommand{\GK}{\operatorname{GK-dim}}
\newcommand{\hd}{\operatorname{hd}}
\newcommand{\id}{\operatorname{id}}
\newcommand{\res}{\operatorname{res}}
\newcommand{\lrar}{\leadsto}
\newcommand{\im}{\operatorname{Im}}
\newcommand{\HH}{\operatorname{H}}
\newcommand{\TF}{\operatorname{TF}}
\newcommand{\Bun}{\operatorname{Bun}}
\newcommand{\Hilb}{\operatorname{Hilb}}
\newcommand{\Fact}{\operatorname{Fact}}
\newcommand{\F}{\mathcal{F}}
\newcommand{\nthord}{^{(n)}}
\newcommand{\Aut}{\underline{\operatorname{Aut}}}
\newcommand{\Gr}{\operatorname{Gr}}
\newcommand{\Fr}{\operatorname{Fr}}
\newcommand{\GL}{\operatorname{GL}}
\newcommand{\gl}{\mathfrak{gl}}
\newcommand{\SL}{\operatorname{SL}}
\newcommand{\ff}{\footnote}
\newcommand{\ot}{\otimes}
\def\Ext{\operatorname {Ext}}
\def\Hom{\operatorname {Hom}}
\def\Ind{\operatorname {Ind}}
\def\bbZ{{\mathbb Z}}

\newcommand{\nc}{\newcommand}
\newcommand{\on}{\operatorname}
\nc{\cont}{\on{cont}}
\nc{\rmod}{\on{mod}}
\nc{\Mtil}{\widetilde{M}}
\nc{\wb}{\overline}
\nc{\wt}{\widetilde}
\nc{\wh}{\widehat}
\nc{\sm}{\setminus}
\nc{\mc}{\mathcal}
\nc{\mbb}{\mathbb}
\nc{\Mbar}{\wb{M}}
\nc{\Nbar}{\wb{N}}
\nc{\Mhat}{\wh{M}}
\nc{\pihat}{\wh{\pi}}
\nc{\JYX}{\cJ_{Y\leftarrow X}}
\nc{\phitil}{\wt{\phi}}
\nc{\Qbar}{\wb{Q}}
\nc{\DYX}{\D_{Y\leftarrow X}}
\nc{\DXY}{\D_{X\to Y}}
\nc{\dR}{\stackrel{\bbL}{\underset{\D_X}{\ot}}}
\nc{\Winfi}{\cW_{1+\infty}}
\nc{\K}{{\mc K}}
\nc{\unit}{{\bf \on{unit}}}
\nc{\boxt}{\boxtimes}
\nc{\xarr}{\stackrel{\rightarrow}{x}}
\nc{\Cnatbar}{\overline{C}^{\natural}}
\nc{\oJac}{\overline{\on{Jac}}}
\nc{\gm}{{\mathbf G}_m}
\nc{\Loc}{\on{Loc}}
\nc{\Bm}{\operatorname{Bimod}}
\nc{\lie}{{\mathfrak g}}
\nc{\lb}{{\mathfrak b}}
\nc{\lien}{{\mathfrak n}}
\nc{\e}{\epsilon}
\nc{\eu}{\mathsf{eu}}

\nc{\Gm}{{\mathbb G}_m}
\nc{\Gabar}{\wb{\G}_a}
\nc{\Gmbar}{\wb{\G}_m}
\nc{\PD}{{\mathbb P}_{\D}}
\nc{\Pbul}{P_{\bullet}}
\nc{\PDl}{{\mathbb P}_{\D(\lambda)}}
\nc{\PLoc}{\mathsf{MLoc}}
\nc{\Tors}{\on{Tors}}
\nc{\PS}{{\mathsf{PS}}}
\nc{\PB}{{\mathsf{MB}}}
\nc{\Pb}{{\underline{\operatorname{MBun}}}}
\nc{\Ht}{\mathsf{H}}
\nc{\bbH}{\mathbb H}
\nc{\gen}{^\circ}
\nc{\Jac}{\operatorname{Jac}}
\nc{\sP}{\mathsf{P}}
\nc{\sT}{\mathsf{T}}
\nc{\bP}{{\mathbb P}}
\nc{\otc}{^{\otimes c}}
\nc{\Det}{\mathsf{det}}
\nc{\PL}{\on{ML}}
\nc{\sL}{\mathsf{L}}

\nc{\ml}{{\mathcal S}}
\nc{\Xc}{X_{\on{con}}}
\nc{\Z}{{\mathbb Z}}
\nc{\resol}{\mathfrak{X}}
\nc{\map}{\mathsf{f}}
\nc{\gK}{\mathbb{K}}
\nc{\bigvar}{\mathsf{W}}
\nc{\Tmax}{\mathsf{T}^{md}}

\title{Compatibility of $t$-Structures for Quantum Symplectic Resolutions}
\author{Kevin McGerty}
\address{Mathematical Institute\\University of Oxford\\Oxford OX1 3LB, UK}
\email{kmcgerty@mac.com}
\author{Thomas Nevins}
\address{Department of Mathematics\\University of Illinois at Urbana-Champaign\\Urbana, IL 61801 USA}
\email{nevins@illinois.edu}

\begin{abstract}
Let $W$ be a smooth complex quasiprojective variety with the action of a connected reductive group $G$.  Adapting the stratification approach of Teleman \cite{Teleman} to a microlocal context, we prove a vanishing theorem for the functor of $G$-invariant sections---i.e., of quantum Hamiltonian reduction---for $G$-equivariant twisted $\D$-modules on $W$.  As a consequence, when $W$ is affine we establish an effective combinatorial criterion for exactness of the global sections functors of microlocalization theory.  When combined with the derived equivalence results of \cite{McN1}, this gives precise criteria for ``microlocalization of representation categories'' in the spirit of \cite{GS1, GS2, Holland, KR, DK, MVdB, BKu, BPW, McN1}.  
\end{abstract}

\maketitle


\section{Introduction and Statement of Results}
\subsection{Introduction}
Many noncommutative algebras of intense recent interest are naturally realized via
quantum Hamiltonian reduction from the ring of differential operators on a smooth complex 
affine variety $W$ with the action of a complex Lie group $G$.  Examples
include deformed preprojective algebras \cite{Holland}, the spherical subalgebras of cyclotomic Cherednik algebras
\cite{Gordon, Oblomkov} and more general wreath product
symplectic reflection algebras \cite{EGGO, Losev2}.  Suitably microlocalized categories of equivariant or twisted-equivariant $\D$-modules---in more sophisticated language, $\D$-modules on stacks---provide a natural tool for categorifying representations of quantum groups (cf. \cite{Zheng, LiA, LiB, Web}); in such terms, the algebra realized by quantum Hamiltonian reduction is the space of global sections of a (similarly suitably microlocalized) sheaf of algebras.

\vspace{.7em}

Quantum Hamiltonian reduction depends naturally on a parameter: namely, a character
$c$ of the Lie algebra $\mathfrak{g}$ of the group $G$.  
Under a precise, effectively computable
combinatorial condition on $c$, we prove a vanishing theorem for the
functor of quantum Hamiltonian reduction of $\D$-modules.  

\vspace{.7em}

Precise statements of our main results and their consequences appear in Section \ref{statement of results} below.  However, the rough form of the main results are as follows.  Suppose $W$ is a smooth, connected quasiprojective complex variety with an action of a connected reductive group $G$.  
Let $\chi: G\rightarrow \Gm$ be a group character.
The group $G$ is equipped with  a finite set $\mathcal{S}$ of 1-parameter subgroups of a fixed maximal torus $\mathsf{T}\subset G$, depending on $W$ and $\chi$: these are the {\em Kirwan-Ness 1-parameter subgroups}.  An algorithm for computing $\mathcal{S}$ when $W$ is a representation of $G$ is explained in the body of the paper.  To each $\beta$ we associate a numerical shift $\mathsf{shift}(\beta)$, defined precisely below, and a subset $I(\beta)\subseteq \mathbb{Z}_{\geq 0}$.
\begin{rough-thm-1}
Suppose that, for each $\beta\in\mathcal{S}$, 
\bd
c(\beta)\notin \mathsf{shift}(\beta) + I(\beta) \subseteq \mathsf{shift}(\beta) + \mathbb{Z}_{\geq 0}.
\ed
Then any $c$-twisted, $G$-equivariant $\D$-module with {\em unstable singular support} is in the kernel of quantum Hamiltonian reduction.
\end{rough-thm-1}
 
\noindent
As a
consequence, we establish effective criteria for $t$-exactness of direct
image functors in microlocalization theory a la \cite{KR, McN1, BPW}, or indeed in any reasonable technical framework for localization results in characteristic $0$.  

Namely, if $W$ is affine, $\mu: T^*W\rightarrow \mathfrak{g}^*$ is the moment map for the $G$-action, $\mu$ is flat, and the GIT quotient $\resol = \mu^{-1}(0)/\!\!/_{\chi} G$ is smooth, there are various technical frameworks to produce a natural quantization $\mathcal{W}_{\resol}(c)$ of $\theo_{\resol}$ depending on $c$ and a ``reasonable'' category of quasicoherent $\mathcal{W}_{\resol}(c)$-modules.\ff{We note that the ``quotient category'' framework of \cite{McN1} works well even when $\resol$ is not smooth.}  We prove:
\begin{rough-thm-2}
Suppose that, for each $\beta\in\mathcal{S}$, 
\bd
c(\beta)\notin \mathsf{shift}(\beta) + I(\beta) \subseteq \mathsf{shift}(\beta) + \mathbb{Z}_{\geq 0}.
\ed
The the functor of global sections on quasicoherent $\mathcal{W}_{\resol}(c)$-modules is exact.
\end{rough-thm-2}
Our results thus provide a far-reaching analogue of the exactness part of
the seminal Beilinson-Bernstein localization theorem in geometric
representation theory, both extending and complementing important precursors \cite{BKu, GGS, GS1, GS2, Holland, KR, MVdB}.  A crucial point is the effectiveness of the combinatorics involved, which provides us with precise control over when such exactness results hold.  We illustrate this effectiveness with a quick and easy derivation of exactness for the quantization of the Hilbert scheme $(\mathbb{C}^2)^{[n]}$ yielding the spherical type $A$ Cherednik algebra.

\subsection{Precise Statement of Results}\label{statement of results}
More precisely, suppose $W$ is a smooth, connected, quasiprojective complex algebraic variety, equipped with the action of a complex reductive group $G$.  A substantial menagerie of interesting examples already arise when $W$ is a representation of $G$; for example, $W$ could be the representation space of a quiver (of a fixed dimension vector) and $G$ the natural automorphism group.  We assume that the canonical line bundle $K_W$ of $W$ is trivialized and is thereby $G$-equivariantly isomorphic to the twist of $\theo_W$ by a character $\gamma_G: G\rightarrow \Gm$---as we explain in Section \ref{Vanishing and Split Surjections for Affine}, this is not a significant restriction.  Write $\rho = \frac{1}{2}d\gamma_G$ and write $\D_W$ for the sheaf of differential operators on $W$, and $\D(W)$ for the algebra of global differential operators.    

Let $\mathfrak{g}= \on{Lie}(G)$.  Associated to the $G$-action there are an {\em infinitesimal $\mathfrak{g}$-action} encoded by a map $\mathfrak{g}\rightarrow \D(W)$, $\mathfrak{g}\ni X\mapsto \wt{X}$, and a {\em canonical quantum comoment map} $\mu^{\on{can}}: \mathfrak{g}\rightarrow \D(W)$, $\mu^{\on{can}}(X) = \wt{X} + \rho(X)$.  Passing to the associated graded and dualizing yields a classical moment map $\mu: T^*W\rightarrow \mathfrak{g}^*$.  
Next, fix a character $c:\mathfrak{g}\rightarrow \C$, or equivalently a linear map $\mathfrak{g}/[\mathfrak{g},\mathfrak{g}]\rightarrow \C$.    Associated to $c$ there is a category of 
{\em (canonically) $c$-twisted $G$-equivariant $\D_W$-modules}: these are $\D_W$-modules $M$ equipped with a $G$-action whose derivative equals the action of $\mathfrak{g}$ via $\mu^{\on{can}}_c: = (\mu^{\on{can}}+c): \mathfrak{g}\rightarrow \D(W)$.  See Section \ref{equivariant modules} for more.  
The category of such modules is denoted $(\D,G,c)-\on{mod}$.  

The {\em quantum Hamiltonian reduction} of $\D_W$ at $c$ is the algebra 
\bd
U_c:= H^0\big(\D_W/\D_W\mu^{\on{can}}_c(\mathfrak{g})\big)^G.
\ed
When $W$ is affine this can be written $U_c = \big(\D(W)/\D(W)\mu^{\on{can}}_c(\mathfrak{g})\big)^G$.
Letting $\mathscr{M}_c = \D_W/\D_W\mu^{\on{can}}_c(\mathfrak{g})$,
the {\em quantum Hamiltonian reduction functor} is 
\bd
(\D,G,c)-\on{mod}\longrightarrow U_c-\on{mod},
\hspace{2em}
M\mapsto \mathbb{H}(M) := \on{Hom}_{(\D,G,c)}(\mathscr{M}_c, M).
\ed
When $W$ is affine, then, writing 
$M_c = \D(W)/\D(W)\mu^{\on{can}}_c(\mathfrak{g})$, 
then $\mathbb{H}(M)$ is equivalently given by
\bd
M\mapsto \mathbb{H}(M)= \on{Hom}_{(\D,G,c)}(M_c, M) \cong M^G.
\ed
The main technical result of the paper is a characterization of a part of the kernel of this functor.   This technical result has strong implications for compatibility of standard $t$-structures in ``microlocalization theory'' for the algebra $U_c$.  

\vspace{0.7em}

To state our results, recall, following the exposition of \cite{Kirwan} and terminology of \cite{Teleman}, the notion of a {\em Kirwan-Ness} (or KN) stratification of the unstable locus of $T^*W$.  The stratification depends on a choice of $G$-equivariant line bundle $\mathscr{L}$ on $T^*W$.  The most interesting examples for us arise when we choose the trivial line bundle with $G$-action twisted by a group character $\chi: G\rightarrow \C$.    As we explain in detail in Section \ref{KN section}, such a stratification of $T^*W$ {\em always} exists  when $W$ is affine, and the $\chi$-unstable locus $T^*W^{\chi-\on{uns}}$ of $T^*W$ thus obtained agrees with that defined more concretely for affine varieties  in \cite{King}.  For more general quasiprojective $W$, we will assume $T^*W$ is equipped with such a stratification (Definition \ref{def:KN-strat}).  

As in \cite{Kirwan, Teleman}, the KN stratification decomposes $T^*W^{\chi-\on{uns}}$ into a finite disjoint union $\coprod S_\beta$ of simpler pieces, each labelled by a $1$-parameter subgroup $\beta$ of $G$.  We give an explicit, constructive recipe for computing the list of 1-parameter subgroups $\beta$ in Section \ref{KN section}; we carry out a sample calculation  relevant to the type $A$ spherical Cherednik algebra in Section \ref{type A spherical section}.   We may (Lemma \ref{KN coisotropic lemma}) restrict attention to a subset 
\bd
\mathsf{KN} = \big\{ \beta \; | \; S_\beta\cap \mu\inv(0) \; \text{contains a nonempty coisotropic subset}\big\}.
\ed

To each connected component of the stratum $S_\beta$ labelled by a 1-parameter subgroup $\beta\in\mathsf{KN}$ we associate three things.  The first is the sum of all the negative $\beta$-weights on $\mathfrak{g}$, which we denote by $\on{wt}_{\mathfrak{n}^-}(\beta)$ (since the corresponding weight spaces span a nilpotent Lie subalgebra $\mathfrak{n}^-$).  To define the second, let $Z_{\beta,i}\subset S_\beta$ denote a connected component of the $\beta$-fixed locus,\footnote{When $W$ is a representation, there is only one $Z_{\beta, i}$ for each $\beta$.} and choose $z\in Z_{\beta, i}$. Then 
$\Gm$ acts via $\beta$ on the normal space $N_{Z_{\beta,i}/T^*W}(z)$ and we write $\on{abs-wt}_{N_{Z_{\beta,i}/T^*W}}(\beta)$ to mean the sum of absolute values of $\beta$-weights on the normal space; it does not depend on the choice of $z\in Z_{\beta, i}$.  Third, $\Gm$ acts on the normal space $N_{S_\beta/T^*W}(z)$ to the stratum $S_\beta$ at $z\in Z_{\beta, i}$, and we let $I_{G,T^*W}(\beta,i)$ denote the set of weights of $\beta$ on the symmetric algebra $\on{Sym}^\bullet\big(N_{S_\beta/T^*W}(z)\big)$---it is a set of non-negative integers (and similarly does not depend on $z\in Z_{\beta, i}$).  

Choosing a filtration of $\D_W$ or, when $W$ is affine, $\D(W)$  yields a notion of the {\em singular support} $SS(M)\subset T^*W$ of a $\D$-module $M$.  For a general $W$ one only knows how to define the {\em operator filtration}, with functions in degree zero and vector fields in degree 1, but in special cases one knows many more: for example, if $W$ is a $G$-representation, any linear $\Gm$-action on $T^*W$ commuting with the $G$-action and having all weights non-positive determines one (Section \ref{quantum notation}).  Fix one of these filtrations.  When $W$ is affine write 
$M_c(\chi^{\ell}) = M_{c-\ell d\chi}\otimes \chi^{\ell}$; as in Formula (4.2) of \cite{McN1}, this is also a $c$-twisted $G$-equivariant $\D$-module.  Assume
we have fixed a choice of character $\chi:G\rightarrow \Gm$ and a corresponding KN stratification of $T^*W$. 
\begin{thm}[Theorem \ref{general affine strong vanishing}]\label{main theorem intro}
Let $W$ be a smooth complex quasiprojective variety.  Fix a character $c:\mathfrak{g}\rightarrow \C$.  Suppose that for every $\beta\in\mathsf{KN}$ and every connected component $Z_{\beta, i}$ of the $\beta$-fixed locus of the KN stratum $S_\beta$, we have 
\begin{equation}\label{intro main formula}
c(\beta)  \notin \left(I_{G,T^*W}(\beta,i) + \on{wt}_{\mathfrak{n}^-}(\beta) + \frac{1}{4}\on{abs-wt}_{N_{Z_{\beta,i}/T^*W}}(\beta)\right).
\end{equation}
Then:
\begin{enumerate}
\item[(i)] If $M$ is any object of $(\D,G,c)-\on{mod}$ with $SS(M)\subseteq (T^*W)^{\chi-uns}$, then $\Hom_{(\D,G,c)}(M_c, M)=0$. 
\end{enumerate}
Suppose that, in addition, $W$ is affine.  Then:
\begin{enumerate}
\item[(ii)]
For every $\ell\ll 0$, there is a finite-dimensional vector subspace 
\bd
V_\ell\subset \Hom_{(\D,G,c)}(M_c(\chi^\ell), M_c)
\ed
 for which the natural composite
evaluation map
\bd
M_c(\chi^\ell)\otimes V_\ell \longrightarrow M_c(\chi^\ell)\otimes \Hom_{(\D,G,c)}\big(M_c(\chi^\ell), M_c\big) \longrightarrow 
M_c
\ed
is a split surjective homomorphism of objects of $(\D,G,c)-\on{mod}$. 
\end{enumerate}
\end{thm}
\noindent
When $W$ is a representation of $G$, one has a more combinatorial statement:
\begin{thm}[Theorem \ref{reductive strong vanishing}]
Suppose $W$ is a representation of $G$.  Let $\alpha_1, \dots, \alpha_d$ be the weights of the maximal torus $\sT\subseteq G$ on a basis $w_1,\dots, w_d$ of $W$.  Suppose that for every $\beta\in\mathsf{KN}$ we have
\bd
c(\beta)  \notin \left(I_{G,T^*W}(\beta) + \on{wt}_{\mathfrak{n}^-}(\beta) + \frac{1}{2}\sum_{i=1}^d |\alpha_i\bullet\beta|\right).
\ed
Then conclusions (i) and (ii) of Theorem \ref{main theorem intro} hold.
\end{thm}

Part (i) of the theorem says that unstably-supported twisted-equivariant $\D$-modules are in the kernel of quantum Hamiltonian reduction, except for a precise collection of values of the twist $c$.  Part (ii) provides a flexible, general tool for proving, in any reasonable framework for a ``microlocalization theory'' for the algebra $U_c$, that the global sections functor is exact---in other words, right exact for the standard $t$-structures on the two categories---provided the hypothesis on $c$ in the theorem is satisfied.\footnote{Condition \ref{intro main formula} is thus the analogue of ``dominant'' in Beilinson-Bernstein localization.}   In Sections \ref{W framework} and 
\ref{McN framework} we make this statement precise in two such frameworks, the deformation quantization approach used in \cite{KR} and the quotient category approach of \cite{McN1}.  The slogan is as follows (all undefined terms are from \cite{KR}).  
\begin{corollary}[Theorems \ref{W thm}, \ref{t-exactness}]
Suppose the condition on $c$ of Theorem \ref{main theorem intro} is satisfied.  Suppose $Z = \mu\inv(0)/\!\!/_{\chi}G$  is a smooth Hamiltonian reduction via a GIT quotient at the character $\chi$ of $G$; let $\cW(c)$ denote the sheaf of deformation quantization algebras on $Z$ constructed by quantum Hamiltonian reduction.  Then the global sections functor for good $\Gm$-equivariant $\cW(c)$-modules is exact.
\end{corollary}
\noindent
The same statement then follows for objects of the ind-category of good $\Gm$-equivariant $\cW(c)$-modules---this is the ``correct'' notion of quasicoherent $\cW(c)$-module for geometric representation theory.  
Theorem \ref{main theorem intro} similarly yields an analogue of the corollary in any other natural framework.  

As an application, we quickly prove (a slightly weakened form of) the exactness part of the microlocalization of \cite{KR} for type $A$ spherical Cherednik algebras in Section \ref{type A spherical section}; since the derived equivalence part of \cite{KR} was handled in \cite{McN1}, this completes a new approach to that problem.  Similarly, calculating KN one-parameter subgroups and applying Theorem \ref{main theorem intro} to the result, one expects an exactness theorem for microlocalization of spherical cyclotomic Cherednik algebras that complements the derived equivalence established in \cite{McN1}, thus yielding an abelian microlocalization theory for those algebras.  Progress in this direction has been achieved by Rollo Jenkins and, separately, Chunyi Li (works in preparation).

In a different direction, one can immediately proceed from our results for quasiprojective varieties to similar results for algebraic stacks.  
We plan to return to this subject elsewhere, so for the moment we only briefly sketch it.
Suppose 
that $\resol$ is a smooth algebraic stack that is exhausted by Zariski-open substacks of the form $W/G$ where each $W$ is a smooth quasiprojective variety and $G$ is a reductive group.  Assume furthermore that $T^*\resol$ comes equipped with a stratification that induces a KN stratification on each $T^*W$ in an appropriate sense.  Our theorem then implies a corresponding vanishing statement for twisted $\D$-modules on $\resol$.  

In particular, fix a smooth projective curve $C$ and let $\on{Bun}_{\mathbb{G}}(C)$ denote the moduli stack of principal $\mathbb{G}$-bundles on $C$ for a reductive group $\mathbb{G}$.  Let $\mathsf{det}$ denote the determinant line bundle on $\on{Bun}_{\mathbb{G}}(C)$.
\begin{corollary}
For all but countably many values of $c$, if $M \in D_{\on{coh}}\big(\D_{\on{Bun}}(\mathsf{det}^{\otimes c})\big)$ has unstable microsupport
in $T^*\on{Bun}_{\mathbb{G}}(C)$ then
\bd
\Hom\big(\D_{\on{Bun}}(\mathsf{det}^{\otimes c}), M\big) = 0.
\ed
\end{corollary}

\subsection{Methods}
The main inspiration for Theorem \ref{main theorem intro} is the elegant proof by Teleman \cite{Teleman} that ``quantization commutes with reduction.'' Teleman's proof uses the KN stratification to reduce to a simple analysis of weights for the $\beta$-action along $S_\beta$.  It was understood clearly by Ian Grojnowski, Kobi Kremnizer, and possibly many others long ago that Teleman's approach should be used to prove a result like Theorem \ref{main theorem intro}.  It was equally clear that the proof cannot reduce simply to weight-space calculations as in \cite{Teleman}, since Theorem \ref{main theorem intro} depends crucially on the parameter $c$ and nothing similar is true in the classical limit.  

The new ingredient beyond \cite{Teleman} is provided by Kashiwara's Equivalence, applied in a more flexible symplectic setting.  Our approach to that adaptation uses a slice theorem to reduce from a full KN stratum $S_\beta$ to its ``Morse-theoretic core'' $Y_\beta$, the locus that attracts to the $\beta$-fixed locus $Z_\beta^{ss}$ under the downward $\beta$-flow.  Although there is a rich and beautiful theory of symplectic slices and symplectic normal forms with a long history (from \cite{GS} to the recent achievements of \cite{Knop, Losev1}), in the case we need---a slice for a free action of a unipotent group---it is easiest to work by hand.  Alternatively, it may be possible to simplify the proof even further using the techniques of \cite{BDMN}.

The details of the symplectic geometry and its quantization are carried out in Sections \ref{symplectic geom} and \ref{DQ}, based on tools from  Section \ref{preliminaries} and a model case, deduced from Kashiwara's Equivalence, in Section \ref{kashiwara section}.  Section \ref{KN section} lays out  basics of KN strata and an algorithm for computing the KN 1-parameter subgroups.  Section \ref{main theorems section} proves the main theorems, and Section \ref{type A spherical section} applies it all to type $A$ spherical Cherednik algebras.

\vspace{1em}

We are grateful to Gwyn Bellamy, David Ben-Zvi, Chris Dodd, Iain Gordon, Ian Grojnowski, Mee Seong  Im, Kobi Kremnizer, Eugene Lerman, Chunyi Li, Tony Pantev, and Toby Stafford for many fruitful and illuminating conversations.  Both authors are grateful to MSRI, and the second author is grateful to All Souls College, Oxford, for excellent working conditions during the preparation of this paper.  
The first author was supported by a 
 Royal Society research fellowship.  The second author was supported by NSF grants DMS-0757987 and DMS-1159468 and NSA grant H98230-12-1-0216, and by an All Souls Visiting Fellowship.  Both authors were supported by MSRI.

\section{Preliminaries}\label{preliminaries}
In this section we lay out some preliminary conventions and facts.  
\begin{basic convention}
Throughout the paper, all varieties are connected (the ground field is always $\C$).  Groups $G$ are assumed to be connected and reductive;  $\sT$ will always denote a torus, typically a maximal torus in an ambient reductive group $G$.  All group actions are assumed to be effective.
\end{basic convention}

\subsection{Group Actions}\label{group actions section}
Suppose a group $G$ acts on a smooth variety $\bigvar$.  For $f\in \C[\bigvar]$, $g\in G$, we let
$(g\cdot f)(x) = f(g\inv x).$ 
Given a character $\chi: G\rightarrow \Gm$, we make the trivial line bundle 
$\mathsf{L} = \bigvar\times {\mathbb A}^1$ into a $G$-equivariant line bundle via $g\cdot (x, z) = (g\cdot x, \chi(g)z)$.  
Recall that a function 
$f: \bigvar\rightarrow {\mathbb A}^1$ is a {\em relative invariant} or {\em semi-invariant} of weight $\chi$ if $f(g\cdot x) = \chi(g) f(x)$ for all $g\in G$ 
and $x\in \bigvar$.  Suppose that $F: \bigvar\rightarrow \mathsf{L}$ is a section, and write 
$F(x) = (x, f(x))$ for a function $f: \bigvar\rightarrow \mathbb{A}^1$.  Then 
$g\cdot F(x) = (gx, \chi(g)f(x))$, and so $F$ is $G$-equivariant if and only if $f$ is $\chi$-semi-invariant. 

\begin{lemma}\label{semi-invariant vs isotypic}
Suppose $\bigvar$ is a smooth variety with $G$-action and $\chi: G\rightarrow \Gm$ is a character.  Then 
a function $f\in\C[\mathsf{W}]$ is $\chi^q$-semi-invariant if and only if $f$ is in the $\chi^{-q}$-isotypic component of $\C[\mathsf{W}]$.
\end{lemma}

\subsection{Differential Operators}\label{diff ops}
Suppose an algebraic group $G$ acts rationally on the smooth affine variety $W$.  Let $\D(W)$ denote the algebra of differential operators on $W$.
For $f\in \C[W]$, $\theta\in \D(W)$, and $g\in G$, we let
$(g\cdot \theta)(f) = g\cdot (\theta(g\inv\cdot f))$.  Differentiating the $G$-action (Section \ref{group actions section}) on $\C[W]$ gives a Lie algebra homomorphism
\begin{equation}\label{general infinitesimal}
\mathfrak{g} = \on{Lie}(G) \rightarrow \Gamma(T_W)\subset \D(W), \hspace{2em} X\mapsto \wt{X},
\end{equation}
the {\em infinitesimal $G$- (or $\mathfrak{g}$-)action}.

If $W$ is a finite-dimensional $G$-representation, then differentiating the homomorphism $G\rightarrow \on{Aut}(W)$ yields a Lie algebra homomorphism 
$\on{act}: \mathfrak{g}\rightarrow \End(W) = W\otimes W^*$.  Writing $\tau: W\otimes W^* \rightarrow W^*\otimes W$, $\tau(w\otimes v) = v\otimes w$, for the canonical braiding, the infinitesimal $\mathfrak{g}$-action on $\C[W] = \on{Sym}(W^*)$ is induced by 
\begin{equation}\label{act dual}
\on{act}^*: \mathfrak{g} \longrightarrow \End(W^*) = W^*\otimes W, \hspace{1em} \on{act}^*(X) = -\sigma\big(\on{act}(X)\big).
\end{equation}
Composing with the canonical map $m: W^*\otimes W\rightarrow \D(W)$, we get $\wt{X} = m\big(\on{act}^*(X)\big)$.

In particular, if $G=\Gm$ and 
 $W=\oplus W_k$ with $W_k$ the $k$-weight space, then for $x\in (W_k)^*\subset\C[W]$ we get $\lambda\cdot x = \lambda^{-k}x$.  Let $\ft = \on{Lie}(\Gm)$.  Then $\ft$ acts infinitesimally on $W$ as follows.  If $v_1,\dots, v_n$ is a basis of $W$ consisting of weight vectors and $\Gm$ acts on $v_i$ with weight $w(v_i)$, then writing $x_i = v_i^*$,
\begin{equation}\label{inf action of 1ps}
\C = \ft \owns {\mathbf 1} \mapsto \wt{\mathbf 1} = \sum -w(v_i) x_i \partial_{x_i}.
\end{equation}

Suppose that a reductive group $G$ acts on a vector space $W$.  
Make a choice of isomorphism $W=\C^N$ under which the maximal torus $\mathsf{T}\subseteq G$ acts by diagonal matrices, and let $\Tmax=\Gm^N$ with the canonical action on $W$ (here the notation for $\Tmax$ is meant to convey ``maximal dimensional'').  We write $\psi_1, \dots, \psi_N$ for the corresponding characters of $\sT$ on $W$ and 
$\alpha_i = d\psi_i$. 
As we will do later in Section \ref{quantum notation}, for any subgroup $K$ of $G$ let $\gamma_{K}: K\rightarrow \Gm$ denote the character of the $K$-action on $\bigwedge^N W^*\cong \C$, and let $\ds \rho_{K} = \frac{1}{2}d\gamma_{K}$: thus, if $K=\sT$,
$\ds\rho_{\sT} = -\frac{1}{2}\sum_{i=1}^N \alpha_i \in \mathfrak{t}^*$.    Define
\begin{equation}\label{eq:def-of-mu-can}
\mu^{\on{can}}_K(X) = \wt{X}+\rho_{K}(X)\in \D(W) \hspace{1em} \text{for $X\in\mathfrak{k} = \on{Lie}(K)$};
\end{equation}
this is the {\em canonical quantum comoment map} for $\D(W)$.  In terms of \eqref{act dual}, 
\begin{equation}\label{action formula for cmm}
\ds \mu^{\on{can}}(X) = m\big(\on{act}^*(X)\big) + \frac{1}{2}\on{tr}\big(\on{act}^*(X)\big).
\end{equation}
When $K=G$ we omit the subscript on $\mu^{\on{can}}$.
If $\beta: \Gm\rightarrow \sT\subseteq G$ is a 1-parameter subgroup of $\sT$ and $e_\beta = \wt{\mathbf{1}}$, 
\begin{equation}\label{canonical moment map formula}
\mu^{\on{can}}(d\beta) = e_\beta - \frac{1}{2}\sum_{i=1}^N \alpha_i\bullet\beta =: \eu(\beta)
\end{equation}
(cf. also \eqref{euler vf}).  For a Lie algebra character $c:\mathfrak{g}\rightarrow\C$, we write $\mu^{\on{can}}_c = \mu^{\on{can}}+c$.

 More generally, suppose $W$ is any smooth variety with trivialized canonical bundle $K_W = W\times\C$.  Suppose $\gamma_K: K\rightarrow \Gm$ is a character such that $k\cdot (w, c) = (k\cdot w, \gamma_K(k)c)$ for all $k\in K$, $w\in W$, and $c\in\C$.  As above, we define 
 $\ds \rho_{K} = \frac{1}{2}d\gamma_{K}$ and $\mu^{\on{can}}$ as in \eqref{eq:def-of-mu-can}.
 
 \begin{remark}\label{compatibility of cans}
 Suppose that $f:W\rightarrow V$ is any $K$-equivariant \'etale morphism of smooth varieties.  Then there is a pullback morphism $f^*: \D(V)\rightarrow\D(W)$ on differential operators, and $f^*\circ \mu^{\on{can}} = \mu^{\on{can}}$.
 \end{remark}

\subsection{Deformation Quantizations}
  We will work with deformation quantization (or {\em DQ}) algebras; an excellent general reference is  \cite{KSDQ}.  If $\bigvar$ is a smooth affine variety with Poisson structure $\{\bullet,\bullet\}$, a 
  DQ algebra structure is an associative, $\hbar$-linear product $\ast$ on $\C[\bigvar][\![\hbar]\!]$ such that 
  \begin{equation}\label{poisson str compatible}
  \ds f\ast g  = fg + \frac{\hbar}{2} \{f,g\} + \theo(\hbar^2).
  \end{equation}
    We will write $\theo^{\hbar}(\bigvar)$ for $\big(\C[\bigvar][\![\hbar]\!],\ast\big)$ when $\ast$ is understood from context.  
\subsubsection{}\label{kontsevich quant}
Recall, more generally, that if $\bigvar$ is a smooth affine algebraic variety then there is a canonical 
``Kontsevich quantization,'' i.e., a bijection between formal Poisson structures $\boldsymbol{\{}-,-\boldsymbol{\}} = \sum_{i\geq 1} \hbar^i \{-,-\}_i$ on $\bigvar$ modulo gauge equivalence and deformation quantizations modulo gauge equivalence.  Moreover, this bijection:
\begin{enumerate}
\item preserves first-order terms, i.e., satisfies \eqref{poisson str compatible} for identified structures;
\item is compatible with pullback by \'etale morphisms;
\item associates to the formal Poisson structure $\hbar\{-,-\}$ on $\mathbb{A}^{2n}$, where $\{-,-\}$ is the Poisson bracket associated to any constant (i.e., translation-invariant) bivector field, the Moyal-Weyl product.
\end{enumerate}
We elaborate on (3) in Section \ref{MW elaboration} below.

\subsubsection{}
Suppose we equip the variety $\bigvar$ with a $\Gm$-action for which the Poisson structure $\{-,-\}$ has weight $\ell$: that is, $m_z^*\{-,-\} = z^\ell\{-,-\}$ where $m_z$ denotes action by $z\in\Gm$.  Then, letting $m_z(\hbar) = z^{-\ell}\hbar$, any $\Gm$-invariant formal Poisson structure defines a $\Gm$-equivariant deformation quantization (or deformation quantization ``with $F$-structure'') as described in \cite[Section~2.3]{KR}.  We say $\ast$ is {\em $\Gm$-equivariant with weight $\ell$}.

  \subsubsection{}
  Suppose the algebraic group $G$ acts on $\bigvar$ preserving a symplectic form $\omega$.  A  {\em (classical) moment map} for the action is a $G$-equivariant map $\mu_G: \bigvar \rightarrow \mathfrak{g}^*$ satisfying, for every $X\in \mathfrak{g}$, 
$\langle d\mu, X\rangle = -i_{\wt{X}}\omega$.  The corresponding {\em classical comoment map} is the pullback on functions,
$\mu^*: \on{Sym}(\mathfrak{g})\rightarrow \C[\bigvar]$.

  Suppose the $\Gm$-equivariant (with weight $\ell$) DQ algebra $\theo^{\hbar}(\bigvar)$ is {\em $G$-equivariant}, i.e., the product $\ast$ is also $G$-equivariant.  Differentiating defines a Lie algebra homomorphism 
  $\alpha: \mathfrak{g} \rightarrow \End_{\C[\![\hbar]\!]}(\theo^{\hbar}(\bigvar))$.  
  \begin{defn}\label{qcomo}
  A {\em quantum comoment map} for the action is a 
$G$-equivariant  linear map $\mu: \mathfrak{g}\rightarrow \theo^{\hbar}(\bigvar)$ satisfying:
  \begin{enumerate}
  \item $[\mu(X), -] = \hbar \cdot\alpha(X)$.
  \item Modulo $\hbar$, $\mu$ becomes a classical comoment map.
  \item For every $X\in\mathfrak{g}$, $\mu(X)$ has $\Gm$-weight $\ell$.
  \end{enumerate}
  \end{defn}

\subsubsection{}\label{std sympl form}
The {\em Liouville 1-form} $\theta$ on the cotangent bundle $\bigvar = T^*W$ is given as follows: if $\pi: T^*W\rightarrow W$ denotes projection, then for a tangent vector $X$ on $T^*W$, $\theta_{\xi}(X) = \xi\big(d\pi(X)\big)$.  This yields a canonical symplectic form $\omega_{T^*W} = d\theta$.  Note the lack of a sign change.  If $W$ is an affine space with coordinates $x_1,\dots, x_n$ and dual cotangent fiber coordinates $y_1,\dots, y_n$ then
$\theta = \sum_i y_i dx_i$ and hence
$\omega_{T^*W} = -\sum dx_i\wedge dy_i.$
For this form, one has the Poisson bracket of coordinate functions $\{x_i, y_j\} = -\delta_{ij}$.

A cotangent bundle also has a {\em canonical classical moment map}, given by dualizing the linear map $\mathfrak{g}\rightarrow \Gamma(T_W)\subset \C[T^*W], X\mapsto \wt{X}$.
\subsubsection{}\label{MW elaboration}
We elaborate on fact (3) from Section \ref{kontsevich quant}.  On an affine space $\bigvar = \mathbb{A}^{2n}$ with translation-invariant symplectic form $\omega$ defining a Poisson structure $\{-,-\}$, one has the {\em Moyal-Weyl product}, defined by the formula:
\begin{equation}\label{moyal}
f\ast g = m\circ e^{\frac{\hbar}{2}\{-,-\}}(f\otimes g).
\end{equation}
This means we view $\{-,-\}$ as a bivector field, $\{-,-\} = \sum_{i,j} \pi_{i,j} \partial_i\wedge \partial_j$ for scalars $\pi_{i,j}$, exponentiate as a bidifferential operator with $\hbar$ coefficients, apply the result to $f\otimes g$, and multiply in $\C[\mathbb{A}^{2n}]$.  The Moyal-Weyl product makes $\C[\mathbb{A}^{2n}][\![\hbar]\!]$ into an associative algebra with unit, flat over $\C[\![\hbar]\!]$, whose truncation mod $\hbar$ is $\C[\mathbb{A}^{2n}]$ and for which $f\ast g - g\ast f = \hbar\{f,g\} +\mathcal{O}(\hbar^2)$.  
 The Moyal-Weyl product is compatible with localization of functions: it makes $\C[\mathbb{A}^{2n}][\![\hbar]\!]$ the global sections of a sheaf of algebras 
 on $\mathbb{A}^{2n}$, and if $\bigvar\subset \mathbb{A}^{2n}$ is an affine open subset the restriction of the Moyal product to $\C[\bigvar][\![\hbar]\!]$ is defined by the same formula \eqref{moyal} for elements $f,g\in\C[\bigvar]$.

\subsection{Filtrations on $\D$}\label{quantum notation}
We discuss filtrations of rings of differential operators.

First assume $W$ is a vector space with linear $G$-action.
Fix a linear $\Gm$-action on $\bigvar = T^*W$, commuting with the $G$-action, that acts with weight $\ell<0$ on the canonical Poisson structure and with nonpositive weights on $\bigvar$ (so nonnegative weights on $\C[\bigvar]$).  We will call a fixed choice of such $\Gm$-action a {\em contracting $\Gm$-action}.  A $\Gm$-stable subset for this action is called {\em conical}.

Such a $\Gm$-action defines a grading on the space $\bigvar^* \subseteq \C[\bigvar]$ of linear functions on $\bigvar = T^*W$.  The Weyl algebra $\D(W)$ is defined by 
\bd
\D(W) = T^\bullet(\bigvar^*)/(x\otimes y - y \otimes x - \{x,y\} \hspace{.3em}|\hspace{.3em} \text{$x, y\in \bigvar^*$ are $\Gm$-weight vectors}).
\ed
The sign of the weight $\ell$ of the Poisson structure guarantees that, if we give $\bigvar^*$ and its tensor algebra the increasing filtration by $\Gm$-weight, the algebra $\D(W)$ inherits a nonnegative filtration, and its associated graded comes equipped with an isomorphism to $\C[\bigvar]$ as graded algebras.  Standard filtrations on $\D(W)$ are
obtained this way, including the Bernstein filtration (in which $\bigvar^*\subset\C[\bigvar]$ has weight one) and the operator filtrations (in which one of
$W^*$ or $W$ has weight one and the other has weight zero).  

If $W$ is an arbitrary smooth, connected variety, we equip $\bigvar=T^*W$ with the $\Gm$-action that contracts cotangent fibers and acts trivially on $W$, and refer to this as the contracting action.  It corresponds to the filtration by order of differential operators on $\D_W$.

\subsection{From $\D$ to DQ}\label{from D to DQ}
In Sections \ref{sec:DQ1}, \ref{sec:DQ2}, and \ref{sec:DQ3}, we assume $\bigvar$ is a symplectic vector space.
 \subsubsection{}\label{sec:DQ1}
There is a close relationship between the Moyal product, when $\omega$ is the standard symplectic form (Section \ref{std sympl form}) on $\bigvar = \mathbb{A}^{2n}$, and the $n$th Weyl algebra $\D= \D(\mathbb{A}^n)$.  Namely, replace the Weyl algebra by its homogenized cousin, defined by 
\bd
\D_{\widehat{\hbar}} = \C\langle x_1, \dots, x_n, y_1, \dots, y_n\rangle[\![\hbar]\!]/([x_i, x_j], [y_i, y_j], [y_i, x_j] - \delta_{ij}\hbar).
\ed
The subalgebra $\D_\hbar$ consists of expressions that are polynomial in $\hbar$.
With these relations, one has the usual identification at $\hbar=1$ with the algebra of differential operators via $\ds y_i\leftrightarrow \frac{\partial}{\partial x_i}$. 

The {\em symmetrization map} 
$\C[x_1,\dots, x_n, y_1, \dots, y_n] \xrightarrow{\on{Symm}} \D_{\widehat{\hbar}}$
 is defined on a monomial $a_1\cdot\dots\cdot a_k$ in the generators $x_i, y_i$ of the polynomial ring by 
\bd
\on{Symm}(a_1\cdot\dots\cdot a_k) = \frac{1}{k!}\sum_{\sigma\in S_k} a_{\sigma(1)}\otimes\dots\otimes a_{\sigma(k)}
\ed
(where the tensor product really means the image of that element in $\D_{\widehat{\hbar}}$). 
\begin{lemma}
The symmetrization map $\on{Symm}$, extended linearly to $\hbar$, intertwines the Moyal $\ast$-product on $\C[x_1,\dots, x_n, y_1, \dots, y_n][\![\hbar]\!]$ with the product on $\D_{\widehat{\hbar}}$.  
\end{lemma}
Note that
\begin{equation}\label{basic Euler op}
\on{Symm}(x_iy_i) = \frac{1}{2} y_ix_i+ \frac{1}{2}x_iy_i = x_i y_i + \frac{\hbar}{2},
\end{equation}
which corresponds to $\ds x_i\frac{\partial}{\partial x_i} + \frac{\hbar}{2}$ under the usual identification of $y_i$ with $\ds\frac{\partial}{\partial x_i}$.  
More generally, for an element $X\in W^*\otimes W$, letting 
\bd
W^*\otimes W \xrightarrow{m} \C[W\oplus W^*], \hspace{2em} W^*\otimes W\xrightarrow{\overline{m}} \D(W)
\ed
denote the multiplication maps, we get
\begin{equation}\label{symm of mm}
\on{Symm}\big(m(X)\big) = \frac{1}{2}\big(m'(X) + m'(\sigma(X))\big) = m'(X) + \frac{\hbar}{2}\on{tr}(X).
\end{equation}

\subsubsection{}\label{sec:DQ2}  We continue with the assumptions of Section \ref{quantum notation}.  Write $\bigvar^* =  \oplus_{\alpha} \bigvar_\alpha^*$ where $\bigvar_\alpha^*$ is the $\Gm$-weight $\alpha$ subspace of $\bigvar^*$.  then the algebra
\begin{equation}\label{presenting Dsubhbar}
T^\bullet\big(\oplus_{\alpha} \bigvar_\alpha^* t^\alpha\big)[t]/\big([ut^{r}, vt^{s}] - \{u, v\}t^{r+s}\hspace{.2em}\big|\hspace{.2em} u \in \bigvar^*_r, v \in \bigvar^*_s\big)
\end{equation}
is isomorphic to $\D_\hbar[\hbar^{1/\ell}]$ via
\bd
ut^{\on{wt}(u)} \mapsto u \hspace{1em} (\text{for $\Gm$-homogeneous $u$}), \hspace{2em} t\mapsto \hbar^{1/\ell}.
\ed

\subsubsection{}\label{sec:DQ3} As in Section \ref{quantum notation}, fix a contracting $\Gm$-action; filter $\D(W)$ by $\Gm$-weight.  We get a corresponding Rees algebra $\cR(\D) = \oplus F_k(\D)t^k \subset \D[t]$.  We can map the tensor algebra $T^\bullet\big(\oplus_{\alpha} \bigvar_\alpha^* t^\alpha\big)[t]$ to $\cR(\D)\subset \D[t]$ in the obvious way; this map is surjective and factors, via the presentation \eqref{presenting Dsubhbar} of $\D_\hbar[\hbar^{1/\ell}]$, through an isomorphism of $\cR(\D)$ with $\D_\hbar[\hbar^{1/\ell}]$.  It follows that $\D_{\widehat{\hbar}}[\hbar^{1/\ell}]$ is the $\hbar$-completion of $\cR(\D)$.

The following is immediate from Formulas \eqref{symm of mm} and \eqref{action formula for cmm}:
\begin{lemma}\label{momentary agreement}
For any character $c: \mathfrak{g}\rightarrow \C$, the map $\cR(\D) \hookrightarrow \D_{\widehat{\hbar}}[\hbar^{1/\ell}] \xrightarrow{\on{Symm}\inv} \theo^{\hbar}(T^*W)[\hbar^{1/\ell}]$ identifies the twisted canonical moment map
$(\mu^{\on{can}}+ c)\in F_\ell\big(\cR(\D)\big)$ with a quantum comoment map for $\theo^{\hbar}(T^*W)$.  The latter equals the  twist 
$\mu_c^{\on{can}}:= \mu + \hbar c$ of the image in $\theo^{\hbar}(T^*W)$ of the canonical classical moment map $\mu\in \C[T^*W]$ under the inclusion 
$\C[T^*W]\hookrightarrow\theo^{\hbar}(T^*W)$.
\end{lemma}

\begin{convention}
Henceforth, throughout the remainder of the paper, we write $\D_{\widehat{\hbar}}$ and $\theo^\hbar(\bigvar)$ to mean 
$\D_{\widehat{\hbar}}[\hbar^{1/\ell}]$ and $\theo^\hbar(\bigvar)[\hbar^{1/\ell}]$, respectively. 
\end{convention} 

We can thus define a canonical quantum comoment map $\mu^{\on{can}}: \mathfrak{g}\rightarrow \theo^{\hbar}(T^*W)$ as the composite of the canonical classical comoment map 
$\mathfrak{g}\xrightarrow{\mu} \C[T^*W]$ followed by the inclusion $\C[T^*W]\hookrightarrow \theo^{\hbar}(T^*W)$; by Lemma \ref{momentary agreement}, this agrees with the canonical comoment map to $\D(W)$ under the natural algebra homomorphisms.

\subsubsection{From $\D$ to DQ for Smooth Varieties}
Finally, we assume $W$ is an arbitrary smooth variety, equip $T^*W$ with the scaling action of $\Gm$, and equip $\D_W$ with the operator filtration.  There is then a canonical choice of deformation quantization $\theo^{\hbar}_{T^*W}$ of the sheaf of functions
on $T^*W$ as a Poisson algebra given by the Kontsevich formula.     Letting $p: T^*W\rightarrow W$ denote projection, the deformation quantization comes equipped with a homomorphism $p\inv \D_W\rightarrow \theo^{\hbar}_{T^*W}$.  If $W^\circ\subseteq W$ is an open set with an \'etale 
morphism $q:W^\circ \rightarrow W'$, then $q$ determines a ``wrong-way'' \'etale morphism $dq: T^*W^\circ \rightarrow T^*W'$.  The pullback $dq\inv \theo^{\hbar}_{T^*W'}\rightarrow \theo^{\hbar}_{T^*W^\circ}$ is an isomorphism of sheaves of associative $\C[\![\hbar]\!]$-algebras.  The algebra $\theo^{\hbar}_{T^*W}$ comes equipped with a canonical splitting of sheaves of vector spaces $\theo_{T^*W}\rightarrow \theo^{\hbar}_{T^*W}$; this splitting is compatible with the map $dq$ induced by an \'etale morphism of varieties $q: W\rightarrow W'$. When $W^\circ\subset W$ is affine and $q:W^\circ\rightarrow\mathbb{A}^n$ is \'etale, then the isomorphism $dq\inv \theo^{\hbar}_{T^*\mathbb{A}^n}\rightarrow \theo^{\hbar}_{T^*W^\circ}$ intertwines the quantization of $T^*W$ with the Moyal-Weyl product on $T^*\mathbb{A}^n$.  

If, in addition, $W$ has trivialized canonical bundle $K_W = W\times\C$ via which $G$ acts by the character $\gamma_G$, then the analogue of Lemma \ref{momentary agreement} holds for $p\inv\D_W\rightarrow \theo^{\hbar}_{T^*W}$.

Slightly abusively, we will write $\D_{\wh{\hbar}}$ for the Rees algebra of the deformation quantization algebra $\theo_{T^*W}^\hbar$ in the above setting.

    \subsection{Equivariant Modules}\label{equivariant modules}
A {\em weakly $\mathbb{K}$-equivariant $\D_W$-module} is a $\D_W$-module $M$ with a rational $\mathbb{K}$-action such that 
$g\cdot(\theta m) = (g\cdot \theta)(g\cdot m), \;\;\text{for all}\;\; m \in M, g \in \mathbb{K}, \theta \in \mathcal D_W$.
The category of such modules is denoted $(\mathcal D_W,\mathbb{K})-\on{mod}$, or if $W$ is affine,
$(\mathcal D(W),\mathbb{K})-\on{mod}$.
   One similarly defines weakly equivariant $\theo^\hbar(\bigvar)$-modules. For notational simplicity, we proceed as if $W$ is affine, though all statements generalize appropriately to non-affine $W$.

Returning to a connected reductive $G$ acting on $W$ and a weakly equivariant $\D(W)$-module $M$, let $\alpha: \mathfrak{g}\rightarrow \End_{\C}(M)$ denote the infinitesimal Lie algebra action.  Given a Lie algebra homomorphism $c:\mathfrak{g}\rightarrow\C$, let
\bd
\gamma_{M,c} = \alpha - (\mu^{\on{can}}+c) : \mathfrak{g}\rightarrow \End_{\C}(M).
\ed
   The module $M$ is {\em $(G, c)$-equivariant} if
$\gamma_{M,c} =0$.  The category of such modules is denoted $(\D, G, c)-\on{mod}$.

We define  $\ds\Phi_c(M) = M/(\sum_{z\in\mathfrak{g}} \gamma_{M,c}(z)M)$ for weakly $G$-equivariant $M$: this yields a $(G,c)$-equivariant module (in either category) and $\Phi_c$ is a left adjoint to the forgetful functor.  We write $M_c(\rho) = \Phi_c(\D\otimes \rho)$ for 
characters $\rho: G\rightarrow \Gm$, and in particular $M_c = M_c(\on{triv}) = \D/\D\mu^{\on{can}}_c(\mathfrak{g})$; this is compatible with the notation $M_c(\chi^{\ell})$ of the introduction by Formula (4.2) of \cite{McN1}.   The functor $\Phi_c$ of course depends on the group $G$, and if we want to emphasize the group we write $M_c^G(\rho)$.  
For a review of the basic properties of twisted equivariant and weakly equivariant $\mathcal D$-modules and these functors, see for example \cite[\S 4]{McN1}, and for a more detailed account \cite{KashiwaraEqui}.

The functor $\Hom_{(\D,G)}(M_c,-)$ is the functor of {\em quantum Hamiltonian reduction}.  Basic properties are discussed (in notation consistent with the present paper) in \cite{McN1}.  In particular, if $M \in (\D,G,c)-\on{mod}$, then
$\Hom_{(\D,G)}(M_c,M) = M^G$.

\subsection{Induction to DQ Modules}
From  Section \ref{from D to DQ}, we get a functor 
\bd
\cR(\D)-\on{mod}\longrightarrow \D_{\widehat{\hbar}}-\on{mod}, \hspace{2em} N\mapsto \D_{\widehat{\hbar}}\otimes_{\cR(\D)} N =: N^{\hbar}.
\ed
In particular,  a $\D$-module $M$ with a choice of good filtration yields a graded Rees module that can naturally be completed to a finitely generated
$\D_{\widehat{\hbar}}$-module, and hence yields a module for the Moyal-Weyl algebra (or when $W$ is not a representation,  a sheaf of modules for the sheaf of DQ algebras) which we will denote by $\cR(M)^{\hbar}$.  
The homomorphism $\D_{\widehat{\hbar}}\rightarrow \D_{\widehat{\hbar}}[\hbar\inv]$ induces a composite functor
\bd
\cR(\D)-\on{gr-mod} \longrightarrow \D_{\widehat{\hbar}}[\hbar\inv]-\on{mod}.
\ed
We also have the usual functors
\bd
\D-\on{mod} \xrightarrow{\C[t,t\inv] \otimes -} \D[t,t\inv]-\on{gr-mod} \leftarrow \cR(\D)-\on{gr-mod},
\ed
where the second functor comes via the identification $\cR(\D)[t\inv] =  \D[t,t\inv]$.  
\begin{lemma}
If $M$ is a finitely generated $\D$-module equipped with a choice of good filtration, the images of $M$ and $\cR(M$) in $\D[t,t\inv]-\on{gr-mod}$ are 
isomorphic.
\end{lemma}
Recall the notion of {\em support} of a finitely generated $\theo^{\hbar}_\bigvar[\hbar\inv]$-module, where $\theo^{\hbar}_\bigvar$ is a DQ algebra.
If $M$ is a finitely generated $\theo^{\hbar}_\bigvar[\hbar\inv]$-module, we choose a finitely generated $\theo^{\hbar}_\bigvar$-submodule 
$M(0)\subset M$ with the property that $M(0)[\hbar\inv] = M$; such a submodule is called a {\em lattice}.  Then, by definition, 
$\on{supp}(M) = \on{supp}(M(0)/\hbar M(0))$, where the latter means the set-theoretic support of the finitely generated $\theo_\bigvar$-module
$M(0)/\hbar M(0)$.  By standard arguments, this notion does not depend on the choice of lattice \cite[Proposition~2.0.5]{Kashiwara}.
\begin{prop}\label{induction to DQ}
\mbox{}
\begin{enumerate}
\item We get a commutative diagram:
\bd
\xymatrix{
\cR(\D)-\on{gr-mod} \ar[d]\ar[r] & \D_{\widehat{\hbar}}-\on{mod}\ar[d] \\
\D[t,t\inv]-\on{gr-mod} \ar[r] & \D_{\widehat{\hbar}}[\hbar\inv]-\on{mod}.
}
\ed
\item If $W$ is a smooth variety with $G$-action, then all functors are compatible with the $G$-actions (i.e. induce a commutative diagram for weakly $G$-equivariant modules).
\item  If $M$ is a $\D$-module equipped with good filtration, then its (singular) support, calculated in any of the above categories, is the same.
\end{enumerate}
\end{prop}

\section{Kashiwara Equivalence and Equivariant Modules}\label{kashiwara section}
In this section we study $\mathcal D$-modules on a $\sT$-representation $W$ (where $\sT$ is a torus). Fix a contracting $\Gm$-action on 
$T^*W$ commuting with the $\sT$-action.

 \subsection{Torus Weights}
Suppose $\Tmax$ is a torus in $GL(W)$ of dimension $\on{dim}(W)$ commuting with $\sT$ and the contracting $\Gm$-action.  
The action of $\sT$ on $W$ may then be viewed as a homomorphism $\rho \colon \sT \to \Tmax$. 
Choose a $\Tmax$-weight basis $e_1, \dots, e_d$ of $W$ (thus $d=\on{dim}(W)$) and let $x_i$ denote the corresponding linear functions and $\partial_i$ the corresponding partial derivatives.  Then the monomial $x^I \partial^J$ has $\Tmax$-weight $J-I$ (in multi-index notation).  In particular, 
$\D(W)^{\Tmax} = \C[x_1\partial_1, \dots, x_d\partial_d]$. 
\begin{lemma}\label{K weight space}
Let $\mathbf{a} = (a_1,\dots, a_d)$ be a weight of $\Tmax$.  Then the $\mathbf{a}$-weight subspace of $\D(W)$ is
$\D(W)^{\mathbf{a}} = \psi \cdot \C[x_1\partial_1, \dots, x_n\partial_d]$
where $\psi = \prod_{a_i<0} x_i^{-a_i} \prod_{a_i\geq 0} \partial_i^{a_i}$.
\end{lemma}

\subsection{Equivariant Kashiwara Equivalence}  
Suppose $V\subset W$ is a $\mathsf{T}$-invariant subspace.  Let $x_1,\dots, x_k$ be linearly independent weight vectors in $W^*\subset \C[W]$ such that $V=W(x_1,\dots, x_k)$ (i.e. such that $V = \{v\; |\; x_i(v) = 0, i=1, \dots, j\}$).  
Write $\partial_i = \partial/\partial x_i$.  We let $\mathsf{T}$ act on $\C[\partial_1, \dots, \partial_k]$ in the natural way (via the identification with a subring of $\on{Sym}(W)$).   We can also extend the natural free, rank $1$ $\C[\partial_1, \dots, \partial_k]$-module structure to make $\C[\partial_1, \dots, \partial_k]$ into a $\C[x_1,\dots, x_k, 
\partial_1, \dots, \partial_k]$-module for which $x_i \partial_j^\alpha = -\alpha\delta_{ij}\partial_j^{\alpha-1}$.

\begin{prop}[Kashiwara's Equivalence]\label{kashiwara}
Let $M$ be a weakly $\mathsf{T}$-equivariant $\D(W)$-module.  If $M$ is supported on $V\subset W$, then $M$ is $\mathsf{T}$-equivariantly isomorphic to 
$\C[\partial_1, \dots, \partial_k]\otimes_{\C} M'$
 where $M'$ is a weakly $\mathsf{T}$-equivariant $\D(V)$-module.
\end{prop}

\subsection{Decomposition With Respect to a 1-Parameter Subgroup}\label{decomposition}
Suppose now we have a one-parameter subgroup $\beta\colon \Gm \to \sT$ of $\sT$ (and hence, if we assume the action of $\sT$ is effective, a one-parameter subgroup of $\Tmax$).
Write $W_+ = W_+(\beta)$, respectively $W_0 = W_0(\beta)$, respectively $W_-= W_-(\beta)$ for the sum of positive, respectively zero, respectively negative weight subspaces of $W$ under $\beta$.  We then have an identification:
\bd
T^*W = (T^*W)_+\times (T^*W)_0\times (T^*W)_-  = (W_+\times W_-^*)\times T^*W_0 \times (W_-\times W_+^*)
\ed
where $W_+\times W_-^*$, respectively $T^*W_0$, respectively $W_-\times W_+^*$, is the positive, respectively zero, respectively negative weight subspace of $T^*W$.

Note that each of the subspaces $W_{\pm},W_0$ (respectively $W^*_{\pm},W_0^*$) is a direct sum of eigenlines $\C e_i$ (respectively $\C e_i^*$ where $\{e_i^*\}_{i=1}^d$ is the dual basis). 
Choose $\Tmax$-weight vector coordinates $x_1, \dots, x_k$ on $W_-$,  and $y_1,\dots, y_j$ on $W_+$.  
The torus $\sT$ acts on $W$ via a list of characters $\psi_i$, $i=1,\dots, j+k+\ell$, where $i=1,\dots, j$ correspond to $W_+(\beta)$, $i=j+1, \dots, j+\ell$ correspond to $W_0(\beta)$, and
$i=j+\ell+1, \dots, j+\ell+k$ correspond to $W_-(\beta)$ (thus we are \textit{not} assuming that the characters $\psi_i$ are distinct).   We write $\alpha_i = d\psi_i$, and abusively write $\alpha_i\bullet\beta = d(\psi_i\circ\beta)$.   
We write
\begin{equation}\label{I(beta)}
I(\beta) = \Big\{\sum_i n_i|\alpha_i\bullet\beta| \; \Big|\; n_i\geq 0\Big\}
\end{equation}
for the set of $\mathbb{Z}_{\geq 0}$-linear combinations of the $|\alpha_i\bullet \beta|$  (cf. Remark \ref{I for torus}).

\subsection{Partial Fourier Transform}
\label{partial fourier}
Suppose $W = W_1 \times W_2$ is a $\mathsf{T}$-invariant direct sum decomposition of $W$. The partial Fourier transform is an isomorphism 
\[
\Psi\colon   \D(W) = \mathcal D(W_1\times W_2) \to  \D(W_1^*\times W_2).
\]
Since $\mathcal D(W) = \mathcal D(W_1)\otimes \mathcal D(W_2)$ it is enough define $\Psi$ on $\mathcal D(W_1)$. Taking coordinates $y_1,\ldots, y_j$ for $W_1$ and $z_1,\ldots z_j$ the corresponding dual coordinates on $W_1^*$, the isomorphism $\Psi$ is given by:
$\Psi(\partial_{y_i}) = z_i$ and $\Psi(y_i)= -\partial_{z_i}.$
Using this isomorphism we get an equivalence between $\mathcal D(W_1\times W_2)-\on{mod}$ and $\mathcal D(W_1^*\times W_2)-\on{mod}$.  Note:

\begin{equation}\label{half-form FT}
\Psi\Big(y_i \partial_{y_i} + \frac{1}{2}\Big)  = -\Big(z_i\partial_{z_i} + \frac{1}{2}\Big).
\end{equation}

\begin{remark}
It would be more invariant to work with a subspace $W_1<W$ and produce a $\mathcal D$-module on the conormal bundle $T_{W_1}^*(W) \cong W_1^*\times (W/W_1)$. This is the microlocalization functor $\mu_{W_1}$ of Kashiwara-Schapira defined for any submanifold $Y$ of a manifold $X$. Since we only need the special case of subspace of a vector space, we have chosen to keep to a more hands-on approach.
\end{remark}

\subsection{Application of Kashiwara to Twisted Equivariant Modules}
The infinitesimal action of $\mathbf{1} \in \text{Lie}(\Gm) = \C$ associated to the action of $\Gm$ via the homomorphism $\beta$ is given by the Euler operator
\begin{equation}\label{euler op}
\wt{\mathbf{1}} =  e(\beta) = -\sum_{i=1}^j w_i y_i\partial_{y_i} + \sum_{i=1}^k u_i x_i \partial_{x_i}\in \D = \D(W), \hspace{1em} \text{where}
\end{equation}
\bd
w_i = \alpha_i\bullet\beta \;\; \text{for $i=1,\dots, j$} \;\; \text{and} \;\; u_i = - \alpha_{j+\ell+i}\bullet\beta \;\; \text{for $i=1,\dots, k$},
\ed
and 
 each $w_i>0, u_i>0$.\footnote{The signs are consistent since $\beta$ acts with positive weights on $W_+$, hence negative weights on its coordinate functions $y_i$, whereas $\beta$ acts with negative weights on $W_-$, hence positive weights on its coordinate functions $x_i$.}  
Write
\begin{align}\label{euler vf}
\mu^{\on{can}}(d\beta(\mathbf{1})) = \eu(\beta) & = - \sum_{i=1}^j w_i \Big(y_i\partial_{y_i}+\frac{1}{2}\Big)  +  \sum_{i=1}^k u_i \Big(x_i \partial_{x_i}+\frac{1}{2}\Big) \hspace{.2in} \text{and}\\
\mu^{\on{can}}(d\beta(\mathbf{1}))  = \eu'(\beta) & =  \hspace{1.1em} \sum_{i=1}^j w_i \Big(z_i\partial_{z_i}+\frac{1}{2}\Big) + \sum_{i=1}^k u_i \Big(x_i \partial_{x_i}+\frac{1}{2}\Big);
\end{align}
these are the {\em half-density--shifted Euler vector fields}.  
Under the isomorphism
\bd
\D = \D(W_+\times W_0\times  W_-) = \D(W_+)\otimes \D(W_0)\otimes \D(W_-)\cong \D(W_+^*)\otimes \D(W_0)\otimes  \D(W_-)
\ed
as above, 
we see from \eqref{half-form FT} that $\eu(\beta)$ gets identified with $\eu'(\beta)$.  
 
Next, given a Lie algebra character 
$c:\ft \rightarrow \C$  (i.e. linear homomorphism) of $\ft = \text{Lie}(\sT)$,
we write:
\begin{equation}\label{module Mc}
M_c^\beta = \D/\D\big(\eu(\beta) + c\bullet\beta\big).
\end{equation}
Under the partial Fourier transform above, our calculations show that $M_c^\beta$
gets identified with 
$\D'/\D'\big(\eu'(\beta) + c\bullet\beta\big)$ where we write $\D' = \D(W_+^* \times W_0 \times W_-)$ to emphasize which coordinates are the base coordinates.

\begin{lemma}
The natural $\Tmax$-action on $\D(W)$ induces structure of  weakly $\Tmax$-equivariant $\D$-module on $M^\beta_c$.
\end{lemma}
\begin{prop}\label{application of kashiwara}
Define $M_c^\beta$ as in \eqref{module Mc} and $I(\beta)$ as in \eqref{I(beta)}.  Suppose $M_c^\beta \xrightarrow{\phi} M$ is a weakly $\Gm$-equivariant (via $\beta$) $\D$-module homomorphism.  
Suppose that $\phi(\mathbf{1})$ is supported on $W_+\times W_-^*$: that is, for each $x_i$ and $\partial/\partial y_i$ (notation as in Section \ref{decomposition}), there is an $N\gg 0$ such that $x_i^N \cdot \phi(\mathbf{1}) = 0$ (respectively, such that 
$(\partial/\partial y_i)^N \cdot \phi(\mathbf{1}) = 0$).  
Then $\phi=0$ if 
\begin{equation}\label{beta condition}
c\bullet\beta\notin \left(I(\beta) + \frac{1}{2}\sum_{i=1}^d |\alpha_i\bullet\beta|\right).
\end{equation}
\end{prop}
\begin{remark}\label{rational ok}
Both sides of the condition on $c$ in the statement of the condition are homogeneous (for positive rational numbers) of degree $1$ in $\beta$.  It follows both that the condition on $c$ does not depend on $\beta$, up to positive rational number multiples, and that the statement of the proposition remains true if we allow $\beta$ to be a {\em rational} 1-parameter subgroup of $\sT$.  
\end{remark}

\begin{proof}
Given a $\Gm$-equivariant $\phi$ as in the statement of the proposition, let $m=\phi(\mathbf{1})$.  Write $\mathsf{c} = c\bullet\beta$.  
Then $m\in M^{\Gm}$ and
$(\eu(\beta) +\mathsf{c})\cdot m=0$ in $M$.  So it suffices to prove that any such element of $M$ is zero.  Suppose that $m$ is such an element;
then 
$\wt{M} = \D(W_+\times W_-) \cdot m\subseteq M$
 is a $\D(W_+\times W_-)$-submodule, and $m\in \wt{M}^{\Gm}$.  Moreover, we have
$\eu(\beta)\in \D(W_+\times W_-)\subseteq \D(W)$.  So we may replace $M$ by $\wt{M}$ and thus assume both that $W_0=0$ and 
that the support conditions of the proposition are satisfied.  Now, applying a partial Fourier transform as discussed above,
we get $\on{supp}(M) \subseteq \{0\}\subset W_+^*\times W_-$ as a $\D'$-module.

Proposition \ref{kashiwara} now tells us that 
$M \cong \C[\partial_{x_1}, \dots, \partial_{x_k}, \partial_{z_1}, \dots, \partial_{z_j}]\otimes_\C M'$
where $M'$ is a representation of $\Gm$.  Recall that, in $\C[\partial_x]$, one has $x\cdot \partial_x^\ell = (-\ell)\partial_x^{\ell-1}$.  
Consequently, $x\partial_x\cdot \partial_x^a = (-a-1)\partial_x^a$ and thus, writing
$\ds \partial^{\mathbf a} = \prod \partial_{x_i}^{a_i}\prod \partial_{z_i}^{a_i'},$
we get 
\begin{align*}
\eu'(\beta)\cdot \partial^{\mathbf a}  & = \sum_{i=1}^j w_i z_i\partial_{z_i}\partial^{\mathbf a} + \sum_{i=1}^k u_i x_i\partial_{x_i}\partial^{\mathbf a} + \frac{1}{2} \Big(\sum_{i=1}^j w_i + \sum_{i=1}^k u_i\Big)\partial^{\mathbf a}\\
& = -\left[\sum_{i=1}^j w_i(a_i' + 1) + \sum_{i=1}^k u_i(a_i +1) - \frac{1}{2} \Big(\sum_{i=1}^j w_i + \sum_{i=1}^k u_i\Big)\right]\partial^{\mathbf a}.
\end{align*}
Thus, a nonzero $\Gm$-invariant $m$ can only be killed by $\eu(\beta) +\mathsf{c} = \eu'(\beta) + \mathsf{c}$ if
\bd
0 = \mathsf{c} - \left[\sum_{i=1}^j w_i\left(a_i' + \frac{1}{2}\right) + \sum_{i=1}^k u_i\left(a_i +\frac{1}{2}\right)\right].
\ed
or 
\bd 
\mathsf{c} = \sum_{i=1}^j|\alpha_i\bullet\beta|a_i' + \sum_{i=j+\ell+1}^{j+\ell+k} |\alpha_i\bullet\beta|a_i +\frac{1}{2}\sum_{i=1}^d |\alpha_i\bullet\beta|,
\ed
where each $a_i, a_i'\geq 0$, as desired.
\end{proof}

\begin{lemma}\label{lifting symbols for beta}
Suppose that $M$ is a weakly $\Tmax$-equivariant $\D(W)$-module generated by $m\in M^{\Tmax}$.  Suppose that $SS(M)\subseteq W_+\times (W_-)^*$.  Then $M$ is also supported set-theoretically on $W_+\times (W_-)^*$, that is, in the sense of Proposition \ref{application of kashiwara}.
\end{lemma}
\begin{proof}
By the singular support hypothesis and  Lemma 2.5.3(1) of \cite{Kashiwara}, for every  $\partial/\partial y_i$ and $N$ sufficiently large, there exists a differential operator $\psi_N$
such that $\psi_N$ has symbol $\partial^N/\partial y_i^N$ and 
$\psi_N\cdot m=0$.
By a standard argument, since $m$ is $\Tmax$-fixed we may assume $\psi_N$ is a $\Tmax$-weight vector, which then is clearly of the same weight
as its symbol.  By Lemma
\ref{K weight space}, it follows that $\psi_N$ lies in 
$\ds \C\left[y_i\frac{\partial}{\partial y_i}\right]\frac{\partial^N}{\partial y_i^N}$
where $\D(W)^{\Tmax}$ is as described at Lemma \ref{K weight space}.  By our conditions on our filtration of $\D$, the only element in this space with symbol
$\frac{\partial^N}{\partial y_i^N}$ is $\frac{\partial^N}{\partial y_i^N}$ itself.  
\end{proof}

\begin{corollary}\label{beta split surjection}  Suppose that
$c$ satisfies \eqref{beta condition} and that $\chi$ is a character of $\Gm$ for which $\chi\circ\beta$ has positive weight.\footnote{Note that the condition holds whenever $\beta$ is a KN 1-parameter subgroup for $\chi$.}
\begin{enumerate}
\item If $\phi: M^\beta_c\rightarrow M$ is a weakly $\Tmax$-equivariant homomorphism and  $M$ has singular support in $W_+\times(W_-)^*$, then $\phi=0$.\item
For every $\ell\ll 0$, there is a finite-dimensional vector subspace 
\bd
V_\ell\subset \Hom_{(\D,\beta(\Gm),c)}(M_c^\beta(\chi^\ell), M_c)
\ed
 for which the natural composite
evaluation map
\begin{equation}\label{evaluation}
M_c^\beta(\chi^\ell)\otimes V_\ell \longrightarrow M_c^\beta(\chi^\ell)\otimes \Hom_{(\D,\beta(\Gm),c)}\big(M_c^\beta(\chi^\ell), M_c^\beta\big) \longrightarrow 
M_c^\beta
\end{equation}
is a split surjective homomorphism of objects of $(\D,\beta(\Gm),c)-\on{mod}$. 
\end{enumerate}
\end{corollary}
\begin{proof}
(1) By Lemma \ref{lifting symbols for beta}, the hypotheses imply that $M$ is set-theoretically supported on $W_+\times(W_-)^*$, that is, in the sense of Proposition \ref{application of kashiwara}.  Thus, applying Proposition \ref{application of kashiwara}, it follows that $\phi(\mathbf{1})=0$, proving (1).

(2) Assume that $\ell \ll 0$.   To simplify notation, write $M_c = M_c^\beta$, $M_c(\chi^\ell) = M_c^\beta(\chi^\ell)$, and $\Gm$ for $\beta(\Gm)$.  
By adjunction we have
\begin{equation}\label{semis}
\begin{split}
\Hom_{(\D,\Gm,c)}\big(M_c(\chi^\ell), M_c\big) \cong 
\Hom_{\D}(\D\otimes\chi^\ell, M_c)^{\Gm}\\
\cong \Hom_{\Gm}\big(\chi^\ell, M_c\big) \cong \Hom_{\Gm}\big(\C, M_c\otimes\chi^{-\ell}\big). 
\end{split}
\end{equation}
The latter space is (non-canonically) isomorphic, via passing to associated graded (for the filtration on $M_c$ induced from the surjection $\D\rightarrow M_c$), to $\big(\C[\mu_\beta\inv(0)]\otimes\chi^{-\ell}\big)^{\Gm}$ where $\mu_\beta$ denotes the moment map for $\beta(\Gm)$.    Applying Lemma \ref{semi-invariant vs isotypic}
with $q=-\ell$,  elements of $\big(\C[\mu_\beta\inv(0)]\otimes\chi^{-\ell}\big)^{\Gm}$ are $\chi^{-\ell}$-semi-invariants for $\Gm$.  For 
$\ell\ll 0$, the common zero locus of such sections is exactly $W_+\times (W_-)^*$.
It follows that, under the identification of \eqref{semis}, the cokernel of the evaluation map
$M_c(\chi^\ell)\otimes \Hom_{(\D,\Gm,c)}\big(M_c(\chi^\ell), M_c\big) \longrightarrow 
M_c$ has singular support in $W_+\times (W_-)^*$.  We may thus choose a finite-dimension subspace
$V_\ell\subset \Hom_{(\D,\Gm,c)}\big(M_c(\chi^\ell), M_c\big)$ such that the cokernel of \eqref{evaluation} also has singular support in $W_+\times (W_-)^*$.  
  
It is evident from their construction that $M_c(\chi^\ell)$ and  $M_c$ are weakly $\Tmax$-equivariant and that the evaluation map $M_c(\chi^\ell)\otimes \Hom_{(\D,\Gm,c)}\big(M_c(\chi^\ell), M_c\big) \longrightarrow 
M_c$ is weakly $\Tmax$-equivariant: note that $\Gm$ acts trivially on $\Hom_{(\D,\Gm,c)}\big(M_c(\chi^\ell), M_c\big)$ but 
$\Tmax$ may not.  However, since the $\Tmax$-action is rational, the subspace $V_\ell$ may be chosen to be a $\Tmax$-stable subspace, and then
\eqref{evaluation} is $\Tmax$-equivariant; assume we have made such a choice.  Let $\phi: M_c(\chi^\ell)\otimes V_\ell \rightarrow M_c$ denote the composite map in \eqref{evaluation}.
It follows from the construction that $\on{coker}(\phi)$ is $(\Gm,c)$-equivariant, weakly $\Tmax$-equivariant, and has singular support
in $W_+\times (W_-)^*$.  Consequently, conclusion (1) implies that $\on{coker}(\phi)^{\Gm}: (M_c(\chi^\ell)\otimes V_\ell)^{\Gm}\rightarrow M_c^{\Gm}$ is zero.  

We next prove that in fact $\on{coker}(\phi) = 0$.  To do this, we apply $(-)^{\Gm} = \Hom_{(\D,\Gm,c)}(M_c,-)$ to the exact sequence
$M_c(\chi^\ell)\otimes V_\ell \xrightarrow{\phi} M_c \rightarrow \on{coker}(\phi)\rightarrow 0.$
Since $\Gm$ is reductive, the sequence remains exact upon applying $\Gm$-invariants, yielding a surjection
$\phi^{\Gm}: \Hom_{(\D,\Gm,c)}(M_c,M_c(\chi^\ell)\otimes V_\ell) \twoheadrightarrow \Hom_{(\D,\Gm,c)}(M_c,M_c)$ of $\End_{(\D,\Gm,c)}(M_c)$-modules.  
An element $\psi \in \Hom_{(\D,\Gm,c)}(M_c,M_c(\chi^\ell)\otimes V_\ell)$ for which  $\phi^{\Gm}(\psi) = \on{Id}_{M_c}$ is a splitting of $\phi$.
   This proves (2).
   \end{proof}

\begin{example}
Consider $W=\C^{n+1}$ with coordinates $x_0, \dots, x_n$.  Let $\beta: \Gm\rightarrow \Gm$ be given by $\beta(z) = z\inv\cdot \on{Id}$.  
Let $\chi(z) = z$, so $d\chi\bullet\beta = -1$.   Write 
\bd c=\big(\ell + (n+1)/2\big)\cdot d\chi, 
\hspace{1em} \text{and} \hspace{1em} 
\eu(\beta) = \sum x_i\partial_{x_i} +(n+1)/2;
\ed
 then $M_c = \D(W)/\D(W)\big(\sum x_i\partial_{x_i} - \ell\big)$.  This is the $\D$-module that descends to $\D(\theo(\ell))$ on $\mathbb{P}^n$---indeed, note that if $f$ is homogeneous of degree $\ell$ then $(\sum x_i\partial_{x_i} - \ell)(f)= 0$, so $H^0(M_c)^{\Gm}$ naturally acts on $H^0(\mathbb{P}^n,\theo(\ell))$.  For this choice of $\beta$, $W_+ = \{0\}$ and $W_+\times W_-^*$ consists of the fiber $\{0\}\times W^*$ over $0\in W$.  Thus, the
proposition states that for $\ell \notin \mathbb{Z}_{\leq -(n+1)}$, there are no weakly equivariant twisted $\D$-modules on $W$ supported over $0$.  Compare to Example \ref{KN proj} for the reason for our choice of signs.
\end{example}

\section{Kirwan-Ness Stratifications}\label{KN section}
In this section we describe the Kirwan-Ness (KN) stratification of a  $G$-variety, and then in the case of a $G$-representation give a description of the strata and an explicit description of the one-parameter subgroups labeling them (which from now on we will refer to as {\em KN one-parameter subgroups}).

\subsection{Review of KN Stratifications}
Let $G$ be a reductive group and $\sT$ a maximal torus of $G$, with $W$ the corresponding Weyl group. Let $Y(\sT) = \on{Hom}(\mathbb G_m,\sT)$ be the group of $1$-parameter subgroups of $\sT$, and $X(\sT) = \on{Hom}(\sT,\mathbb G_m)$, the group of characters. Let $Y_\mathbb Q = Y(\sT)_\mathbb Q$ be the $\mathbb Q$-vector space\footnote{Since $\sT$ is a torus, $Y_\mathbb Q$ is a rational form of $\on{Lie}(\sT)$. For the reductive group $G$, it makes sense to define the set of ``rational $1$-parameter subgroups'' but it no longer has the structure of a $\mathbb Q$-vector space.}  $Y(\sT)\otimes_\mathbb Z \mathbb Q$ and $X_\mathbb Q = X(\sT)\otimes_\mathbb Z \mathbb Q$ similarly. Let $q\colon Y_{\mathbb Q} \to \mathbb Q$ denote a $W$-invariant, integral, positive definite quadratic form, and $d$ the induced metric. The quadratic form allows us to identify $X_\mathbb Q$ and $Y_{\mathbb Q}$, which we do henceforth. 

Suppose $\bigvar$ is a smooth $G$-variety equipped with a $G$-equivariant line bundle $\mathscr{L}$.  Let $\mathsf{KN} = \{\beta\}$ be a finite collection of 1-parameter subgroups of $\sT$.  
We define a partial order $<$ on $\mathsf{KN}$ by setting $\beta < \beta'$ if $q(\beta)<q(\beta')$. 
Given $\beta\in\mathsf{KN}$, write $Z_\beta = \bigvar^{\beta(\Gm)}$ for the fixed-point locus of $\beta$.  Let 
\bd
\ds Y_\beta = \{x\in \bigvar \;|\; \lim_{t\rightarrow 0} \beta(t)\cdot x \in Z_\beta\}, \;\; \text{and write}\;\;
\on{pr}_\beta: Y_\beta\rightarrow Z_\beta
\ed
for the corresponding projection.  Write $\bigvar^{ss}$ for the open subset of $\bigvar$ complementary to the common zero locus of $G$-invariant elements of $H^0(\bigvar, \mathscr{L}^N)$ for all $N>0$.  

Let $L_\beta \subseteq G$ denote the centralizer of $\beta$ in $G$, and let $P_\beta$ denote the parabolic subgroup of $G$ whose Lie algebra is spanned by the nonnegative $\beta$-weight spaces in $\mathfrak{g}$; in particular, 
$L_\beta$ is the Levi factor of $P_\beta$.  Then $L_\beta$ preserves $Z_\beta$, $P_\beta$ preserves $Y_\beta$, and $\on{pr}_\beta$ is $P_\beta$-equivariant where $P_\beta$ acts on $Z_\beta$ via $P_\beta\rightarrow L_\beta$.  
Moreover, for each connected component $Z_{\beta, i}$ of $Z_\beta$ and $x\in Z_{\beta, i}$, $L_\beta$ acts via a character $\lambda_{\beta, i}: L_\beta \rightarrow \Gm$ on the fiber $\mathscr{L}(x)$ of $\mathscr{L}$ over $x$ (and this character does not depend on the choice of $x\in Z_{\beta, i}$).

\begin{defn}\label{def:KN-strat}
A {\em KN stratification} of $\bigvar$ with respect to $\mathsf{KN}$ consists of a choice, 
 for each $\beta$, of an $L_\beta$-stable open subset $Z_\beta^{ss}\subseteq Z_\beta$ satisfying:
 \begin{enumerate} 
\item for each $\beta$ and each component $Z_{\beta, i}$ , the complement $Z_{\beta, i}\smallsetminus Z_\beta^{ss}$ in $Z_{\beta, i}$ is cut out by a collection of $\lambda_{\beta, i}$-semi-invariant elements of $H^0\big(Z_{\beta, i}, \mathscr{L}^N|_{Z_{\beta, i}}\big)$ for some $N \gg 0$.  
\item Defining $Y_\beta^{ss} = \on{pr}_\beta\inv(Z_\beta^{ss})$, $Y_\beta^{ss}$ is a $\beta$-equivariant affine bundle over $Z_\beta^{ss}$ via $\on{pr}_\beta$.
\item Letting $S_\beta = G\cdot Y_\beta^{ss}$, we have $S_\beta \cong G\times_{P_\beta} Y_\beta^{ss}$.
\item  The collection of subsets $\{\bigvar^{ss}\}\cup \{S_\beta\; |\; \beta\in\mathsf{KN}\}$ stratifies $\bigvar$:
\begin{enumerate}
\item The stratum closures $\overline{S}_{\beta}$ satisfy $\overline{S}_{\beta}\subseteq \bigcup_{\beta'\geq \beta} S_{\beta'}$.\ff{It is \textit{not} assumed to be the case that the closure of a stratum is a union of strata.}
\item  $\bigvar = \bigvar^{ss}\coprod \big(\coprod_\beta S_\beta\big)$.
\end{enumerate}
\end{enumerate}
\end{defn}

\subsection{Sources of KN Stratifications}
Suppose $\wt{X}\subseteq \C^{n+1}$ is the affine cone of a $G$-stable subvariety $X\subseteq \mathbb{P}^n$ of a projective space with linear $G$-action, with coordinates $x_0, \dots, x_n$ which are $\sT$-weight vectors, with weights $\alpha_0, \dots, \alpha_n$.  
Given a $1$-parameter subgroup $\beta:\Gm\rightarrow\sT$, define subsets $\wt{Z}_\beta$ and $\wt{Y}_\beta$ of $\mathbb P^n$ by declaring that a point $(x_0:\dots: x_n)\in\mathbb{P}^n$ lies in $\wt{Z}_{\beta}$, respectively $\wt{Y}_{\beta}$, if and only if, for each $i$ such that $x_i\neq 0$, we have $\alpha_i\bullet \beta = \beta\bullet\beta$, respectively $\alpha_i\bullet \beta \geq \beta\bullet\beta$.  Let
\begin{equation}\label{eq:beta-proj}
Z_\beta = \wt{Z}_\beta\cap X, \hspace{1em} Y_\beta = \wt{Y}_\beta \cap X, \hspace{1em} \text{and write}\hspace{1em} \on{pr}_\beta: \wt{Y}_\beta\rightarrow \wt{Z}_\beta
\end{equation}
for the natural linear projection given by setting to zero those coordinates for which $\alpha_i\bullet \beta > \beta\bullet\beta$.  

The space $\wt{Y}_\beta$ is an affine bundle over the space $\wt{Z}_\beta$. There is a natural open subset $\wt{Y}_\beta^{\on{ss}}$ of $\wt{Y}_\beta$ for which the $1$-parameter subgroup $\beta$ is \textit{optimally destabilizing}, which is the preimage under the bundle map of an open set $\wt{Z}_\beta^{\on{ss}}$ in $Z_\beta$; see \cite[\S 12]{Kirwan} for more details.
The {\em Kirwan-Ness stratum} (or KN stratum) $\wt{S}_\beta$ of $\mathbb{P}^n$ associated to $\beta$ is the $G$-saturation $\wt{S}_\beta = G\cdot \wt{Y}_\beta^{\on{ss}}$; Kirwan proves that $\wt{S}_\beta = G\times_{P_\beta} \wt{Y}_\beta^{ss}$ for the parabolic $P_\beta\subset G$ whose Lie algebra is spanned by nonnegative $\beta$-weights.  
Let $S_\beta = \wt{S}_\beta\cap X$.  
When $G$ is a torus we have $S_\beta = Y_\beta^{\on{ss}}$.

By \cite[\S 12.8]{Kirwan}, the 1-parameter subgroups $\beta$ all arise as follows.\footnote{Strictly speaking, in the terminology of \cite[\S 12]{Kirwan}, the rational $1$-parameter subgroup $\beta/q(\beta)$ is the optimally destabilizing subgroup, but we follows the conventions of \cite{Kirwan} in labeling strata by $\beta$ rather than $\beta/q(\beta)$.}  One takes a nonempty subset $\mathbf{\alpha}$ of the weights $\{\alpha_i\}$, forms their convex hull $\on{conv}(\mathbf{\alpha})$ in $X_{\mathbb{Q}}$, and takes the closest point $\beta$ to $0$ in $\on{conv}(\mathbf{\alpha})$.
We define a partial order $<$ on the set of KN 1-parameter subgroups by setting $\beta < \beta'$ if $q(\beta)<q(\beta')$.  We will need: 

\begin{prop}[\cite{Kirwan}, \S 12.16]
\label{Kirwan stratification}
Suppose that $X\subseteq \mathbb{P}^n$ is a smooth, closed subvariety.  Then $X^{ss}$ together with the strata $S_\beta$ form a KN stratification of $X$.
\end{prop}

We remark that every $1$-parameter subgroup is conjugate under $G$ to a $1$-parameter subgroup in $\sT$. 
In particular, we use the following (cf. \cite[\S 13]{Kirwan}).

\begin{enumerate}
\item
Any KN stratum $S$ contains a point $x$ whose set of optimal $1$-parameter subgroups $\Lambda_G(x)$ intersects $Y(T)$ in a unique point $\beta$. 
\item
If $Y_\beta$ denotes the KN stratum containing $x$ with respect to the $\sT$-action on $V$, then $S = G \cdot Y_\beta$. Thus the $W$-orbit of the point $\beta$ uniquely determines the stratum $S$. 
\end{enumerate}

It follows immediately that once we fix a maximal torus $\sT$ the KN strata are labeled by calculating the KN $1$-parameter subgroups associated to minimal combinations of those weights of the $\sT$-action lying in a single Weyl chamber.

We also have the following elementary fact:
\begin{lemma}\label{lem:induced-KN}
Suppose that $X$ is a smooth $G$-variety with KN stratification and that $\bigvar\subseteq X$ is a smooth $G$-stable locally closed subvariety that is a union of KN strata.  Then the induced stratification of $\bigvar$ is a KN stratification.
\end{lemma}

\subsection{KN Stratification of a Representation}\label{sec:KN-rep}
We are particularly interested in the case where the $G$-variety $X$ is a linear representation $V$ of $G$; to do this, we need the relative version of the theory of KN strata from 
\cite[Sections~1~and~5]{Teleman}.  We consider the vector space $V$ as an open subset of $\mathbb P(\C\oplus V)$ where $\C$ denotes the trivial representation of $G$. We get an action of $G$ on 
\bd
\bP = \bP(\C\times V) = \on{Proj}\; \on{Sym}\big((\C\oplus V)^*\big)
\ed
by $g\cdot (c,v) = (c, gv)$. This action makes $V\cong \{1\}\times V\subset \bP$ into a $G$-invariant open subset of $\bP$ consisting of semistable points. Now fix a character $\chi: G\rightarrow \Gm$ and consider the projective morphism $\pi: \bP\rightarrow \bP$ given by the identity map.  Equip $\pi$ with the relatively ample line bundle $M=\theo$, but with the $G$-linearization twisted by $\chi$.  More precisely, this means the following.  The total space of $M$ is $\bP\times\aline$, and 
we let $g\in G$ act by $g\cdot ((c, v), w) = ((c, gv), \chi(g)w)$ (cf. Section \ref{group actions section}).

   For each $\epsilon\in {\mathbb Q}_{>0}$, we get a $G$-linearized ${\mathbb Q}$-line bundle $\sL = M^{\otimes\epsilon}\otimes\theo(1)$.  Over $V \cong \bP\big((\C\sm\{0\})\times V\big)\subset \bP$, the bundle $\theo(1)$ has a canonical nonvanishing $G$-invariant section, given as a linear form on $(\C\oplus V)^*$ by the projection $\C\oplus V^*\twoheadrightarrow \C$.  In terms of this $G$-invariant trivialization of $\theo(1)$, then, $G$ acts on $\sL = M^{\otimes\epsilon}\otimes\theo(1)$ via a twist by the rational character $\lambda^{\epsilon}$ (see \cite{Teleman} for the conventions about the meaning of such statements).

    On restriction of $\sL\inv$ to the open set $V$ of $\bP$,  this trivialization of $\theo(1)$ allows us to write the action on $\sL\inv$ as  $g\cdot (v, w) = (gv, \lambda^{-\epsilon}(g)w)$;  this formula agrees with the conventions used in \cite{King} for applying geometric invariant theory to a representation, and allows us to use Kirwan's set-up to directly
describe the KN stratification of $V$ for the rational character $\chi^{\epsilon}$.
  \begin{lemma}
The KN stratification of $\mathbb{P}$ induced by $\sL$ is independent of $\epsilon$ for small $\epsilon\in {\mathbb Q}_{>0}$.  The induced stratification of 
$V\subseteq \bP$ is a KN stratification of $V$.
\end{lemma}
\begin{proof}
The first statement follows from 
Lemmas 1.2 and 5.1 of \cite{Teleman}.  The second statement then follows from Lemma \ref{lem:induced-KN}.
\end{proof}
\subsection{KN 1-Parameter Subgroups for a Representation}
We maintain the notation of Section \ref{sec:KN-rep}.
For convenience, however, we now switch to additive notation, so $\lambda =d\chi$; also, to avoid a profusion of minus signs, we assume from now on that $\e$ a small \textit{negative} rational number, so $\sL = M^{\otimes(-\epsilon)}\otimes\theo(1)$.

Let $\{e_1,e_2,\ldots,e_n\}$ be a basis of $V$ consisting of $\sT$-weight vectors, and let $\{\alpha_1,\alpha_2,\ldots,\alpha_n\}$ be the corresponding list of weights (thus we do not assume that the $\alpha_i$ are pairwise distinct). We extend this list by setting $\alpha_0=0$, so that if $e_0$ is the standard basis vector of $\C\times\{0\}\subset \C\oplus V$, then $\{\alpha_0,\ldots,\alpha_n\}$ gives a list of the weights of $\C\oplus V$. A point $x \in \mathbb P(\C\oplus V)$ determines a subset $I_x \subset \{1,2,\ldots,n\}$ where if $x = [t_0:t_1:\ldots: t_n]$ in the homogeneous coordinates given by the basis $\{e_i: 0\leq i \leq n\}$ of $\mathbb C\oplus V$, then 
\[\label{strata}
I_x = \{i \in \{0,1,\ldots,n\}: t_i \neq 0\}.
\]
Note that if $x \in V \subset \mathbb P(\C\oplus V)$, then $0 \in I_x$.

We will need the following standard notation from convex geometry. Let $S$ be a nonempty subset of $W$ a $\mathbb Q$-vector space. We write $\on{conv}(S)$ for the convex hull of $S$, \textit{i.e.}, the set
\[
\on{conv}(S) = \{\sum_{i =1}^m t_is_i: m \in \mathbb Z_{>0}, s_i \in S, t_i \in \mathbb Q, 0 \leq t_i\leq 1, \sum_{i=1}^n t_i =1\}.
\]
Similarly we let $\on{aff}(S)$ denote the affine hull of $S$, that is 
\[
\on{aff}(S) = \{\sum_{i =1}^m t_is_i: m \in \mathbb Z_{>0}, s_i \in S, t_i \in \mathbb Q, \sum_{i=1}^n t_i =1\}.
\]
\begin{remark}
\label{affineandlinear}
Clearly, if we pick any $s_0 \in \on{aff}(S)$, then $\on{aff}(S)$ is the translate by $s_0$ of the subspace spanned by $\{s-s_0: s\in S\}$.  Hence $\on{aff}(S)$ is a linear subspace precisely when $0 \in \on{aff}(S)$, in which case $\on{aff}(S) = \on{span}(S)$. 
\end{remark}

Given a subset $\emptyset\neq I \subseteq \{0,1,\ldots,n\}$, let $W_I= \on{span}\{\alpha_i: i \in I\}$ be the subspace spanned by the corresponding set of weights, and $C_I = \on{conv}\{\alpha_i: i\in I\}$, $A_I = \on{aff}\{\alpha_i: i\in I\}$. We also write $\alpha^\e_i = \alpha_i +\epsilon \lambda$ (so that in particular, $\alpha^\e_0 = \epsilon\lambda$), and then let $A^\e_I = \on{aff}\{\alpha^\e_i: i \in I\}$ and $C^\e_I = \on{conv}\{\alpha^\e_i: i \in I\}$.

The KN $1$-parameter subgroups for $\mathbb P(\C\oplus V)$ are given as follows.  If $I \subset \{0,1,\ldots,n\}$  then let $\beta_I = \beta_I(\epsilon)$ denote the closest point in $C^\e_I$ to $0$, that is, $\beta_I(\e)$ minimizes the distance $d(0,C^\e_I)$. If $\beta_I(\e) =0$ then there are semistable points with exactly this subset of weights; the semistable locus corresponds to the trivial $1$-parameter subgroup. Otherwise, the corresponding (rational) KN $1$-parameter subgroup is $\beta_I/q(\beta_I)$. We write $\Lambda_V$ for the set of KN $1$-parameter subgroups which label KN strata intersecting $V \subset \mathbb P(\C\oplus V)$.  

Note that for each nonempty subset $I \subseteq \{0,1,\ldots,n\}$, the distance $d(0,C^\e_I)$ is a piecewise linear function of $\e$, as is the minimal vector $\beta_I(\e)$. Since there are finitely many such subsets $I$, there is some $C<0$ such that all the distance functions $d(0, C^\e_I)$ and minimal vectors $\beta_I(\e)$ are linear for $\epsilon\in[C,0]$; fix one such $C$.  
Given $I$, fix a {\em minimal} subset $J \subseteq I$ with the property that $d(0,C_I^\e) = d(0,C_J^\e)$ for $\epsilon \in [C,0]$.  Then, since $q$ is strictly convex, so that the closest point is unique, we must have $\beta_I(\e) = \beta_J(\e)$ for $\epsilon\in [C,0]$.  

For each subset $I$, set $p_I(\e) = \on{proj}_{W_I^\perp}(\epsilon\lambda)$, the orthogonal projection to the subspace of weights perpendicular to $W_I$. 

\begin{prop}
\label{relevant beta}
Let $x \in V \subset \mathbb P(\mathbb C\oplus V)$ and let $I = I_x \subset \{0,1,\ldots, n\}$ be the corresponding set of weight labels. If $J\subseteq I$ is minimal such that $d(0,C_{I}^\e)=d(0,C_J^\e)$ then $\beta_{I}(\e) = p_J(\e)$. 
\end{prop}

\begin{proof}
Fix $J\subset I$ minimal so that $\beta_I(\e) = \beta_J(\e)$ for $\epsilon\in[C,0]$. Note that $C_J^\e$ is a closed subset of $A_J^\e$ with nonempty interior \textit{as a subset of} $A_J^\e$; we write $\on{rel-int}(C_J^\e)$ for this relative interior. We claim that 
\begin{equation}\label{open}
\beta_I(\e) \in \on{rel-int}(C_J^\e) \;\; \text{for} \;\; \e \in [C,0].
\end{equation}

Indeed, the boundary of $C_J^\e$ in $A_J^\e$ is the union of convex hulls $C_{J'}^\e$ corresponding to certain proper subsets $J'\subset J$. If $\beta_I(\e)$ lies in one of these boundary subsets $C_{J'}^\e$  for small $\e$ then 
we must have $\beta_{J'}(\e) = \beta_I(\e)$, contradicting the minimality of $J$.

Next we claim that $0 \in C_J$. To see this, let $\beta'$ be the closest point to $0$ in $C_J$. Then as $\beta_J(\e)$ is the closest point to $0$ in $C_J^\e$ and $C_J^\e = \epsilon\lambda +C_J$, we must have 
\begin{equation}
\label{primeclosest}
\|\beta_I(\e)\| \geq \|\beta'\| + \epsilon\|\lambda\|
\end{equation}
(recall that $\epsilon$ is small and negative). On the other hand, since the stratum $S_\beta$ intersects $V$, we must have $0 \in I$, and hence $\alpha_o^\e = \epsilon \lambda \in C_I^\e$, and as $\beta_J(\e) = \beta_I(\e)$ is the unique closest point, by convexity we have:
\begin{equation}
\label{betaclosest}
\epsilon\lambda\bullet \beta_J(\e) \geq \beta_J(\e)\bullet\beta_J(\e) = \|\beta_J(\e)\|^2.
\end{equation}
For $\e$ sufficiently small, the inequalities (\ref{primeclosest}) and (\ref{betaclosest}) can only be consistent if $\|\beta'\| =0$, that is, if $\beta'=0$, and hence $0 \in C_J$ as required. Note in particular by Remark \ref{affineandlinear} this implies $A_J$ is a subspace.

It follows by (\ref{open}) that $\beta_J(\e)$ is also the closest point to $0$ in $A^\e_J = \e\lambda + A_J$. It follows immediately that $\beta_I(\e) = p_J(\e)$ as required.

\end{proof}

\begin{remark}
If $\mathbb P(\C\oplus V)$ has a $KN$ stratum with corresponding one-parameter subgroup $\beta$ of the form $\beta = p_J(\e) = \e \on{proj}_{W_J^\perp}(\lambda)$, then $S_\beta$ intersects $V$. Indeed if the subset $J$ is such that the set $C_J$ does not contain $0$, for small enough $\e$ the set $C_J^\e$ clearly will not contain $p_J(\e)$, contradicting the fact that $p_J(\e)$ should be the closest point to $0$ in $C_J^\e$; so $0\in C_J$. But then we may replace $J$ by $J \cup \{0\}$ without altering $C_J^\e$, and then clearly we may find $v \in V$ with associated weights giving $J$, and hence $v \in S_\beta$ as required. 
\end{remark}


So far we have shown that the set of KN $1$-parameter subgroups is given by a subset of the elements $\{p_J(\e): J \subset \{0,1,\ldots,n\}\}$. We now show how to refine this to a precise description of the KN $1$-parameter subgroups. For this we use an elementary result in convex geometry.

\begin{lemma}
\label{witnesses}
Let $W$ be a finite dimensional $\mathbb Q$-vector space and let $S\subset W$ be a finite set, and $C= \on{conv}(S)$ its convex hull. Then a point $p$ lies in $C$ if and only if, for every other point $p'$, there is some $s \in S$ such that $d(s,p)<d(s,p')$.
\end{lemma}
\begin{proof}
By translating if necessary we may assume that $p=0$. If $0 \in C$, we may write 
\[
0 = \sum_{i=1}^n t_is_i, \;\; t_i \in \mathbb Q\cap [0,1], \sum_{i=1}^n t_i =1.
\]
for some $s_i \in S$ ($1\leq i \leq n$). Now it follows that
\[
0 = p'\bullet 0 = \sum_{i=1}^n t_i p'\bullet s_i.
\]
Since every $t_i$ satisfies $t_i \geq 0$, and not all $t_i$ can be zero since $\sum_{i=1}^n t_i =1$, we must have some $s_i$ with $p'\bullet s_i \leq 0$. But then since $p' \neq 0$ we have
\[
\|p'-s_i\|^2 = \|p'\|^2 -2p'\bullet s_i +\|s_i\|^2 > \|s_i\|^2
\]
and so $d(p',s_i) >d(p,s_i)$ as required. 

Conversely, if $p \notin C$, then let $p'$ be the closest point in $C$ to $p$. Since the closest point is unique, we see that $d(p',s) < d(p,s)$ for all $s \in S$ as required.
\end{proof}

\begin{corollary}
Let $p_I(\e)$ be as above. Then $p_I(\e)$ yields a KN $1$-parameter subgroup if and only if there is no $J \subset I$ with the property that $d(p_J(\e),\alpha_k)<d(p_I(\e),\alpha_k)$ for all $k \in I$.
\end{corollary}
\begin{proof}
One just needs to check that $\e p_I(\lambda)$ lies in $C_J^\e$.  It follows from the above lemma and proposition that for all sufficiently small $\e$ some $\e p_J(\lambda)$ is the closest point in $C_J^\e$ to $0$; the result follows immediately. 
\end{proof}

\begin{example}
Suppose that $V=\C e_1 \oplus \C e_2$ is two-dimensional with $\sT = \Gm^2$ acting with weights $\alpha_1 = (1,0)$ and $\alpha_2 =(1,1)$. Then if $\lambda = (0,1)$, and $\epsilon$ is small and negative, the modified weights on $\C \times V$ are $\{(0,\epsilon),(1,\epsilon),(1,1+\epsilon)\}$. Taking the various possible subsets of weights associated to $x = (x_0:x_1:x_2)\in \mathbb P(\C\times V)$ with $x_0\neq 0$, we find that $\epsilon\lambda$ is the closest point for both subsets $\{(0,\epsilon)\}$ and $\{(\epsilon,0), (0,1+\epsilon)\}$ while $\frac{\epsilon}{2}(1,-1) = \epsilon\lambda^{\perp}$, where the $\perp$ is taken with respect to the subspace $\C(1,1)$, is the closest point for the subsets $(0,\epsilon),(1,1+\epsilon)\}$, $\{(0,\epsilon),(1,\epsilon),(1,1+\epsilon)\}$.

Note however, that it is not the case that all the vectors $\epsilon \lambda^\perp$ as $U$ runs over the span of subsets of the weights $\{\alpha_1,\ldots, \alpha_k\}$ will be minimal combinations of weights: in the above case $\beta=0$ does not occur, showing the semistable locus of this $V$ is empty. Moreover, if we take $\lambda = (0,-1)$, then for small negative $\epsilon$ the minimal weight is always $\epsilon\lambda$ itself. 
\end{example}

\subsection{Explicit Description of KN Strata for a Representation}\label{sec:KN-strata-of-rep}
We next want to describe more explicitly the loci $Z_\beta$ and $Y_\beta$ of Kirwan for $\beta = p_I(\epsilon)$ as above.

Fix $I = \{i_0,i_1, \ldots i_k\} \subseteq \{0,1,\ldots, n\}$ and let $W=W_I$ and $\beta = \beta_I(\e)$. Write $\lambda^\perp = \on{proj}_{W^\perp}(\lambda)$, so $p_I(\epsilon) = \e \lambda^\perp$.  Then $\lambda = \lambda^\perp + \on{proj}_W(\lambda)$ and $\lambda^\perp\bullet \on{proj}_W(\lambda) = 0$, so $\lambda\bullet\lambda^\perp = \lambda^\perp\bullet \lambda^\perp$.  It follows that
\bd
\alpha_i'\bullet \epsilon\lambda^\perp = (\alpha_i+\epsilon\lambda)\bullet \epsilon\lambda^\perp = \alpha_i \bullet \epsilon\lambda^\perp +
\epsilon\lambda^\perp\bullet \epsilon\lambda^\perp,
\ed
and so for $\beta = \epsilon\lambda^\perp$,  Kirwan's conditions become:
\begin{equation}
\text{$(x_0:\dots :x_n)\in Z_\beta$ iff, for each $i$ such that $x_i\neq 0$, we have $\alpha_i\bullet \beta= 0$.}
\end{equation}
\begin{equation}
\text{$(x_0:\dots :x_n)\in Y_\beta$ iff, for each $i$ such that $x_i\neq 0$, we have $\alpha_i\bullet \beta \geq 0$.}
\end{equation}

We introduce some notation for these subspaces.  Given a weight $\beta$, let $V_+(\beta)$, respectively $V_0(\beta)$, respectively $V_-(\beta)$, denote the sum of the $\alpha$-weight subspaces for which $\alpha\bullet\beta >0$, respectively $\alpha\bullet\beta = 0$, respectively $\alpha\bullet\beta <0$; this is equivalent to splitting $V$ according to whether the (rational) 1-parameter subgroup labelled by $\beta$ acts with positive, zero, or negative weight.  We observe that then:
\begin{equation}\label{strats}
Z_\beta = V_0(\beta) \;\;\; \text{and}\;\;\; Y_\beta = V_+(\beta)\times V_0(\beta).
\end{equation}
We also find that $Z_\beta^{ss} = V_0(\beta)^{ss}$ is an open subset of $V_0(\beta)$, and hence that 
\begin{equation}
Y_\beta^{ss} = V_+(\beta)\times V_0(\beta)^{ss}.
\end{equation}

\begin{example}\label{KN proj}
Let $V = T^*W=T^*\C^{n+1}$ with torus $\sT = \Gm$ acting with weight $1$ on $W$.  Let $\chi: \Gm\rightarrow \Gm$ be the identity character (i.e. 
$\chi(z) = z$).  Let $\epsilon$ be a small {\em negative} number.  Then under the $\chi^{-\epsilon}$-twisted linearization, the weights on $\C\times W\times W^*$ are 
$(\epsilon, 1+\epsilon, \dots, 1+\epsilon, -1+\epsilon, \dots, -1+\epsilon)$. By Proposition \ref{relevant beta}, the only relevant
minimal combination of weights is $\beta = \epsilon$, a small negative multiple of the identity character, and $Y_\beta$ consists of all vectors whose $\sT$-weights pair non-negatively with $\beta$, i.e. whose $\sT=\Gm$-weights are nonpositive: in other words, 
$Y_\beta = \{0\}\times W^*$.   
\end{example}

We now specialize to the case $V=T^*W$ where $W$ is a representation of $G$ of dimension $d$.
  We pick a $\sT$-weight basis $\{e_1,\ldots, e_d\}$ of $W$, so that its dual basis $\{\xi_1,\ldots,\xi_d\}$ is a weight basis of $W^*$, and their union is a weight basis of $T^*W$. 
If $\beta$ is a destabilizing 1-parameter subgroup, we then have by definition $Y_\beta = W_+(\beta) \times W_-(\beta)^* \times Z_\beta$, where 
$Z_\beta = T^*W_0(\beta) = W_0(\beta)\times W_0(\beta)^*$.  In particular:
\begin{lemma}\label{Y_beta is coisotropic}
Each $Y_\beta$ is coisotropic in $T^*W$.  Furthermore, 
\bd
Y^{ss}_\beta  = W_+(\beta)\times W_-(\beta)^* \times Z_\beta^{ss}.
\ed 
\end{lemma}

\subsection{KN Stratification of an Affine Variety}
Suppose that $\bigvar$ is a smooth affine $G$-variety.  Fix a character $\chi: G\rightarrow \Gm$.  We may $G$-equivariantly  embed $\bigvar$ as a closed subvariety of a representation $V$ of $G$ (cf. for example \cite{Kempf}, Lemma 1.1).  Section \ref{sec:KN-strata-of-rep} describes the subsets $Z_\beta$ and $Y_\beta$ of $V$ determining the KN strata of $V$
as in \eqref{strats} for each 1-parameter subgroup $\beta\in\mathsf{KN}$.  Thus, $Z_\beta\cap \bigvar = \bigvar^{\beta(\Gm)}$, and 
$\ds Y_\beta \cap \bigvar = \big\{x\in \bigvar \;\big|\; \lim_{t\rightarrow 0}\beta(t)\cdot x\in \bigvar^{\beta(\Gm)}\big\}$. 
\begin{prop}\label{prop:KN-affine-var}
Let $\mathsf{KN}$ denote the set of KN 1-parameter subgroups for $V$.  Let $\wt{S}_\beta$ denote the KN stratum of $V$ corresponding to $\beta$.  Let $S_\beta = \wt{S}_\beta\cap \bigvar$ denote the corresponding locally closed subset of $\bigvar$ and let $\bigvar^{ss} = V^{ss}\cap\bigvar$.  Then $\{\bigvar^{ss}\}\cup\{S_\beta\; |\; \beta \in \mathsf{KN}\}$ is a KN stratification of $\bigvar$.
\end{prop}
\begin{proof}
The only condition in the definition to check is that $Y_\beta\cap\bigvar\rightarrow Z_\beta\cap\bigvar$ is an affine bundle.  By the above description, however, this follows from 
Theorem 4.1 of \cite{BB}.
\end{proof}

  \subsection{Parameter Shifts}\label{parameter shifts}
Suppose $V$ is a representation of $G$.  Given a 1-parameter subgroup $\beta$ of $\mathsf{T}\subseteq G$,  we write $Z_\beta$ for the $\beta$-fixed subspace of $V$.  We let $\overline{Y_\beta}$ be the subspace of $V$ spanned by $\beta$-weight vectors with positive $\beta$-weight. 
Let $K\subseteq G$ be a subgroup containing $\beta$.  Suppose $Y_\beta$ is an open subset of $\overline{Y_\beta}$ whose intersection with $Z_\beta$ is nonempty and for which the orbit space $S_\beta = K\cdot Y_\beta$ is smooth; this will be true in the examples we need in the paper (where $S_\beta$ will be a KN stratum for a KN 1-parameter subgroup $\beta$).  We then 
define the set of nonnegative rationals $I_V(\beta)$ as follows.  Choose a point $z\in Z_\beta\cap Y_\beta$, which is then a fixed point of $\beta$ in $S_\beta$.  Then $\beta$ acts linearly on the normal space $N_{S_\beta/V}(z)$, and we let $\mathsf{w}_1,\dots, \mathsf{w}_a$ denote the $\beta$-weights on this normal space.  Then we write
\bd
I_{K,V}(\beta) = \Big\{\sum_{i=1}^a n_i|\mathsf{w}_i| \; \Big|\; n_i\geq 0\Big\}
\ed
for the set of nonnegative integer linear combinations of the absolute values of $\beta$-weights $|\mathsf{w}_i|$.  This definition extends immediately to rational 1-parameter subgroups, i.e. formal expressions $\ds\frac{p}{q}\beta$ where $\ds\frac{p}{q}$ is rational, by letting 
$\ds I_{K,V}(\frac{p}{q}\beta) = \frac{p}{q}I_{K,V}(\beta)$.   
\begin{remark}\label{I for torus}
Note that if $K=\mathsf{T}$ then $S_\beta = Y_\beta$ and $I_{\mathsf{T},V}(\beta)$ is just the set of nonnegative integer linear combinations of absolute values of $\beta$-weights on $V$.
\end{remark}

\section{Equivariant Symplectic Geometry Near a KN Stratum}\label{symplectic geom}
In this section, $W$ is a smooth, connected quasiprojective variety with a rational action of a connected reductive group $G$ with maximal torus $\sT\subseteq G$.   Let $\mu: T^*W\rightarrow\mathfrak{g}^*$ denote the canonical classical moment map.  We assume that $T^*W$ is equipped with a KN stratification and that the line bundle $\mathscr{L}$ is trivialized so that its $G$-linearization is given by a character $\chi: G\rightarrow\Gm$.

\subsection{}\label{sec:first-sympl}
Recall from Section \ref{KN section} that a KN stratum $S_\beta$ is labelled by (the Weyl group orbit of) a 1-parameter subgroup $\beta: \Gm\rightarrow \sT\subseteq G$ in a fixed maximal torus $\sT$ of $G$.  As in Section 12 of \cite{Kirwan}, $\beta$ determines a parabolic subgroup $P_\beta$ of $G$: letting $L_\beta$ denote the centralizer of $\beta$ in $G$, we let $P_\beta$ denote the subgroup whose Lie algebra is spanned by $\on{Lie}(L_\beta)$ and the positive $\beta$-weight subspaces in $\mathfrak{g}$.  The sum of positive $\beta$-weight subspaces is a nilpotent Lie sub-algebra $\mathfrak{n}\subset \mathfrak{g}$; we let $\mathfrak{n}^-$ denote the opposite nilpotent subalgebra, the sum of negative $\beta$-weight subspaces in $\mathfrak{g}$, and let $U^- = U_{P_\beta}^-\subset G$ denote the corresponding unipotent subgroup of $G$. 

We will write 
\bd
\gK = U^-\rtimes \Gm,
\ed
where $\Gm$ acts on $U^-$ via $\beta$ and the adjoint action of $G$.
The group $\gK$ acts naturally on $T^*W$ via $G$. We view $U^-$ as a $\gK$-variety where $\beta(\mathbb G_m)$ acts by conjugation and $U^-$ by left translation. 

\begin{lemma}
\label{action map}
The action map $a\colon U^-\times Y_\beta \rightarrow T^*W$ is a $\mathbb K$-equivariant bijection onto an open dense subset of $S_\beta = G\cdot Y_\beta \cong G\times_{P_\beta} Y_\beta$. Moreover, $S_\beta$ is coisotropic.
\end{lemma}
\begin{proof}
The equivariance of $a$ is immediate. The isomorphism $S_\beta \cong G\times_{P_\beta} Y_\beta$ is established in Theorem 13.5 of \cite{Kirwan}. That $a$ is bijective follows from this and the fact that $U^-\cdot eP_\beta$ is the open Bruhat cell in $G/P_\beta$. Lemma \ref{Y_beta is coisotropic} asserts that $Y_\beta$ is coisotropic.  Since $S_\beta$ is $G$-stable and $S_\beta = G\cdot Y_\beta$, to check that $S_\beta$ is coisotropic, it suffices to check that the tangent space of $S_\beta$ is coisotropic at each point of $Y_\beta$.  But any subspace of a symplectic vector space that contains a coisotropic subspace is itself coisotropic.
\end{proof}

\subsection{Construction of A Slice}
We will now construct a slice to the action of $U^-$ at a point $z \in Z^{ss}_\beta$.

Suppose $z\in Z_\beta^{ss}$.  The infinitesimal $U^-$-action induces an injective (by Lemma \ref{action map}) map $\mathfrak{n}^-\rightarrow T_z(T^*W)$, and we get a direct sum $\mathfrak{n}^- \oplus T_z Y_{\beta} \subset T_z(T^*W)$.  Since $\beta$ acts on $U^-$ and hence compatibly on $\mathfrak{n}^-$, and $z$ is a $\beta$-fixed point, it makes sense to ask whether $\mathfrak{n}^-\rightarrow T_z(T^*W)$ is $\beta$-equivariant; it clearly is.  Hence the subspace $\mathfrak{n}^- \oplus T_z Y_{\beta}\subset T_z(T^*W)$ is $\beta$-invariant.   

Choose a $\beta$-invariant complementary subspace $V$, so 
$T_z(T^*W) = \mathfrak{n}^- \oplus T_z Y_{\beta} \oplus V$. 
  If $W$ is a $G$-representation, define 
  \bd
  N:= V\times Y_\beta = V\times Z_\beta^{ss}\times (T^*W)_+ \subset T^*W.
  \ed
  More generally, if $W$ is a smooth quasiprojective variety, let $W^\circ\subseteq W$ be a $\beta$-stable affine open subset containing $z$; one exists by \cite[Corollary~3.11]{Su}.  Further shrinking $W^\circ$ if necessary, let 
  \begin{equation}
  q: W^\circ\rightarrow \mathbb{A}^n
  \end{equation}
   be a $\beta$-equivariant \'etale map from a $\beta$-stable affine open subset $W^\circ\subseteq W$ with $q(z) = 0$, where $\mathbb{A}^n$ is a linear representation of $\Gm$; one exists by the \'Etale Slice Theorem (see p. 198 of \cite{Mumford}).  Since the map $q$ is \'etale, it induces a ``wrong way'' cotangent map  map $dq: T^*W^\circ \rightarrow T^*\mathbb{A}^n$ (note the slightly abusive notation). Defining $Z^\dagger_\beta = (T^*\mathbb{A}^n)^{\Gm}$ and $Y^\dagger_\beta\subset T^*\mathbb{A}^n$ to be the $\Gm$-attracting locus of $Z_\beta^\dagger$, then $Z_\beta\cap T^*W^\circ$ is a connected component of  $dq\inv(Z^\dagger_\beta)$ and $Y_\beta\cap T^*W^\circ$ is a connected component of $dq\inv(Y^\dagger_\beta)$.  The tangent map $d(dq_z): T_z(T^*W)\rightarrow T_0(T^*\mathbb{A}^n)$ is an isomorphism.  Abusively writing $V = d(dq_z)(V)\subset T_0(T^*\mathbb{A}^n) = T^*\mathbb{A}^n$, we get 
 \bd
 T^*\mathbb{A}^n = T^*_0\mathbb{A}^n = \mathfrak{n}^- \oplus V\oplus T_0 Y^\dagger_{\beta} = \mathfrak{n}^-\times V\times Y^\dagger_\beta.
 \ed
   Let $N$ denote the connected component of $dq\inv(V\times Y^\dagger_\beta) \subset T^*W^\circ\subseteq \bigvar$ containing $Y_\beta\cap T^*W^\circ$.  Note that $Y^\dagger_\beta$ is coisotropic, hence so is $V\times Y^\dagger_\beta$, hence so is $N$.
  
 Write $\on{pr}_\beta: Y_\beta\rightarrow Z_\beta$ for the projection as in \eqref{eq:beta-proj}. 
\begin{prop}\label{existence of U_D}
For any $z \in Z_\beta^{ss}$ there are an affine neighborhood $D\subset Z_\beta^{ss}$ and a principal open subset $U_D\subset N$ (i.e. the complement of a hypersurface) such that
\begin{enumerate}
\item $U^-$ acts infinitesimally transversely to $U_D$; 
\item $\on{pr}_\beta\inv(D)\subseteq U_D \subseteq N$; and 
\item $(U^-\cdot U_D) \cap S_{>\beta} = \emptyset$. 
\item The complement $N\smallsetminus U_D$ is the hypersurface defined by a $\beta$-invariant function in $\C[N]$.
\item $U_D$ is coisotropic.
\end{enumerate}
The sets $D$ and $U_D$ can be chosen so that:
\begin{enumerate}
\item[(i)] $D$ and $U_D$ are $\beta$-stable. 
\end{enumerate}
 Furthermore, making, for each $z\in Z_\beta^{ss}$, a choice of any affine $D_{z}\subset Z_\beta^{ss}$ containing $z$  and any $U_{D_z}$ satisfying the conditions above,
 \begin{enumerate}
\item[(ii)] The union $\ds \bigcup_{z\in Z_\beta^{ss}} U^-\cdot U_{D_z}$  covers $U^-\cdot Y_\beta$. 
\end{enumerate}
\end{prop}
We will use the following in the proof of the proposition. 
\begin{lemma}\label{hypersurface intersection}
Suppose $\mathsf{W}$ is an affine variety with torus $\mathsf{T}$ acting and $\mathcal{Z}\subseteq \mathsf{W}$ is a closed $\mathsf{T}$-stable subset.  Then $\mathcal{Z}$ is an intersection of $\mathsf{T}$-stable hypersurfaces in $\mathsf{W}$.
\end{lemma}
\begin{proof}[Proof of Proposition \ref{existence of U_D}]
Note that  $N$ has complementary dimension to $\mathfrak{n}^-$.   Pick a principal open subset $D'\subseteq Z_\beta^{ss}$ containing $z$ that is $\beta$-invariant (this is possible since 
$Z_\beta^{ss}$ 
is defined by the nonvanishing of a collection 
of $L_\beta$-semi-invariants)
and define $D\subseteq D'$ to be the open subset consisting of points $p\in D'$ at which the composite linear map 
\begin{equation}\label{full rank map}
\mathsf{inf}(p): \mathfrak{n}^-\rightarrow T_p(T^*W) \twoheadrightarrow T_p(T^*W)/T_pN
\end{equation}
is an isomorphism (where the first map is the infinitesimal action at $p$). 
Since we are asking whether the linear map $\mathsf{inf}(p)$ has vanishing top exterior power, the complement of $D$ in $D'$ is the zero locus of a single function $\wedge^{\on{top}}\mathsf{inf}$, i.e. a principal affine open $D$ of $Z_\beta$ along which $U^-$ acts infinitesimally transversely to $N$.  Now we can define $U_D$ to be the locus of those $p\in N$ at which \eqref{full rank map} is an isomorphism; this is again principal.  
This establishes part (1).

The $\gK$-equivariance of the action map $a$ ensures that the nonvanishing locus of the function $\wedge^{\on{top}}\mathsf{inf}$ is $\beta$-invariant and open.  Now every point of $\on{pr}_\beta\inv(D)$ has a point of $D\subset  U_D$ in its $\beta$-orbit closure, establishing part (2).  

Thus, it suffices to restrict $D$ further, if necessary, to ensure (3).  But $S_{>\beta}$ is $G$-invariant;
 hence its preimage under the restricted action map $a: U^-\times  U_D\rightarrow T^*W$ above is of the form $U^-\times \cZ$ for some closed and 
$\beta$-invariant subset 
$\cZ\subseteq U_D$.  Since $z\notin S_{>\beta}$, we have $z\notin \cZ$.   Replace $U_D$ by a $\beta$-invariant principal affine open in $U_D\smallsetminus \cZ$ that contains $z$ (this is possible by Lemma \ref{hypersurface intersection}) and replace $D$ by $U_D\cap Z_\beta^{ss}$ for this new choice of $U_D$.  Note that if $w\in \on{pr}_\beta\inv(D) \smallsetminus U_D$, then by $\beta$-invariance of $U_D$ and closure of its complement we have 
\bd
\on{pr}_\beta(w) = \underset{t\rightarrow 0}{\lim}\, \beta(t)\cdot w \in Y_\beta \smallsetminus U_D.
\ed
  Hence if $\on{pr}_\beta(w)\in D$ then $w\in U_D$.  Thus these new choices of $D$ and $U_D$ satisfy both (2) and (3), as desired.  
  
   Let $H = N \smallsetminus U_D$, a hypersurface in $N$.  As it is $\beta$-stable, it must be defined by a $\beta$-semi-invariant function.  Note, however, that if $f$ is any $\beta$-semi-invariant function with nonzero $\beta$-weight, then, since $D$ consists of $\beta$-fixed points, $f(D) = 0$; hence $H$ must be 
 defined by a $\beta$-invariant element of $\C[N]$. This proves (4). It is clear from the construction that (i) and (ii) are satisfied.  We have already observed that $N$ is coisotropic, hence so is its open subset $U_D$, proving (5).
 \end{proof}
\begin{corollary}\label{restriction of semi-invariants}
Keep notation as in Proposition \ref{existence of U_D}.  Then, for each weight $k$, the natural map of $\beta$-semi-invariants of weight $k$
\bd
\C[U^-\times N]^{\beta(\Gm), k}\otimes \C[U^-\times U_D]^{\beta(\Gm),0}\rightarrow \C[U^-\times U_D]^{\beta(\Gm),k}
\ed
is surjective.
  In particular, the closed subset of $U^-\times U_D$ defined by the vanishing of $\beta$-semi-invariants of weight $k$  for $k\gg 0$ is 
$U^-\times Y_\beta \subset U^-\times U_D$.
\end{corollary}
\begin{proof}
This is immediate from Proposition \ref{existence of U_D} (4).
\end{proof}

\subsection{Symplectic Geometry Near the Stratum}
Now choose an open subset $U_D\subseteq N \subseteq T^*W$ as in Proposition \ref{existence of U_D} and consider the restricted action map $a:U^-\times U_D\rightarrow T^*W$.  Recall that $a$ is $\gK$-equivariant. 
By Proposition \ref{existence of U_D}(1), $a$ is injective on tangent spaces at points of $U_D$, hence by $U^-$-equivariance of $a$ it is injective on tangent spaces at all points.  The pullback $a^*\omega_{T^*W}$ is a symplectic form on $U^-\times U_D$ which we will describe.

First, we define a closed two-form on $U^-\times U_D$ as follows.  Let $\omega_r$ denote the pullback of the symplectic form on $T^*W$ along the inclusion $U_D\hookrightarrow T^*W$; by abuse of notation, we also write $\omega_r$ for its pullback to $U^-\times U_D$ along the projection on the second factor.  Let $m: U_D\rightarrow (\mathfrak{n}^-)^*$ denote the pullback of the $U^-$-moment map $\mu: T^*W\rightarrow(\mathfrak{n}^-)^*$ along the inclusion $U_D\hookrightarrow T^*W$.    The $1$-form $dm$ is naturally $(\mathfrak{n}^-)^*$-valued; thus, it makes sense to pair it with the canonical $\mathfrak{n}^-$-valued $1$-form $\on{fib}$ on $U^-$ that comes from the left-invariant identification of $T U^-$ with $U^-\times \mathfrak{n}-$, to get a $2$-form $\langle \on{fib}\wedge dm\rangle$.  Explicitly, if $(v_i,w_i)\in T_u U^- \times T_p U_D$ are tangent vectors, 
\bd
\langle \on{fib}\wedge dm\rangle\big((v_1,w_1), (v_2,w_2)\big) = 
\langle v_1, dm_p(w_2)\rangle - \langle v_2, dm_p(w_1)\rangle.
\ed
\begin{prop}\label{form of symplectic form}
The form $\omega_r + \langle \on{fib}\wedge dm\rangle$ is $\gK$-invariant and satisfies 
\bd
\omega_r + \langle \on{fib}\wedge dm\rangle = a^* \omega_{T^*W}.
\ed  
\end{prop}
\begin{proof}
The $\gK$-invariance is immediate from the construction of $\omega_r + \langle \on{fib}\wedge dm\rangle$.
The formula will follow from $\gK$-equivariance of $a$ if we check the equality along $\{e\}\times U_D$.  To do this, observe that $da$ induces isomorphisms on tangent spaces, and that the image of $T_e U^- = \mathfrak{n}^-$ under $da_{(e,p)}$ is isotropic for $\omega_{T^*W}$ since the action is Hamiltonian.  Observe also that
\bd
\langle \on{fib}(v_i), dm_p(w_j)\rangle  = \langle \on{fib}(v_i), d\mu_{p}(w_j)\rangle = 
d\mu_p^{v_i}(w_j) = \omega_{T^*W}(\wt{v}_i, w_j)
= \omega_{T^*W}(da_{e,p}(v_i,0), w_j)
\ed
(cf. \eqref{general infinitesimal} for the notation $\wt{v}_i$).  Expanding $\omega_{T^*W}\big(da_{(e,p)}(v_1,w_1), da_{(e,p)}(v_2,w_2)\big)$ using these two observations gives the desired formula.
\end{proof}

\begin{prop}\label{symplectic embedding prop}
Embed $U^-\times U_D\xrightarrow{\Phi} T^*U^- \times T^*W$ as follows.  Let $m: U_D\rightarrow (\mathfrak{n}^-)^*$ denote the restriction of the $U^-$-moment map from $T^*W$ to $U_D$.  Define 
\bd
\Phi(u,p) = (u, m(p), p)\in U^-\times (\mathfrak{n}^-)^* \times T^*W \cong T^*U^- \times T^*W.
\ed
Then $\Phi$ is a symplectic embedding of $U^-\times U_D$: that is, giving $T^*U^- \times T^*W$ the product symplectic structure $\Omega$, we have $\Phi^*\Omega = a^*\omega_{T^*W}$.
\end{prop}
\begin{proof}
Similar to the previous proposition.
\end{proof}

We have that $U_D$ embeds in $N$.  
\begin{lemma}\label{moment regularity}
The moment map $m = \mu|_{U_D}$ is regular: in particular, $m^{-1}(0)$ is a smooth subvariety of $U_D$.   Moreover, $m(D) = 0$, 
i.e., $D\subseteq Z_\beta^{ss}\cap m\inv(0)$.  
\end{lemma}
\begin{proof}
If $d\mu_p|_{T_pU_D}$ does not have full rank (i.e., linearly independent component functions) for some $p \in U_D$, then there exists $0\neq X\in\mathfrak{n}^-$ with $d\mu^X|_{T_pU_D}\equiv 0$.  The vector field $\wt{X}$ on $T^*W$ generated by $X$ thus satisfies $\omega_{T^*W}(\wt{X}_p, -) \equiv 0$ on $T_pU_D$.  Thus $\wt{X}_p\in (T_pU_D)^\perp$ (with respect to the symplectic form).  Since $U_D$ is coisotropic, we get $\wt{X}_p\in T_pU_D$,   contradicting infinitesimal transversality of the $U^-$-action with respect to $U_D$.  This proves
the first statement.

To see that $Z_\beta^{ss}$ lies in the zero preimage of the moment map, note that the moment map at a point $(z_1,z_2) \in Z_\beta \cong T^*(W^{\beta(\Gm)})$ is the fiberwise dual to the infinitesimal action map $\mathfrak{n}^-\rightarrow T_{z_1}W$.  Since this map is 
$\beta$-equivariant and $\mathfrak{n}^-$ has only negative $\beta$-weights, we find that it factors via $\mathfrak{n}^-\rightarrow (T_{z_1}W)_-\hookrightarrow T_{z_1}W$.  Consequently, the dual map $T_{z_1}W^*\rightarrow (\mathfrak{n}^-)^*$ factors through $(T_{z_1}W^*)_-$; in particular, it kills $z_2\in (T_{z_1}W^*)_0$, and thus $\mu(z_1,z_2)=0$.
\end{proof}

\subsection{Coisotropic Reduction}
Consider the \'etale chart $dq:T^*W^\circ\rightarrow T^*\mathbb{A}^n$ and the 
 linear subspace $N^\dagger:= V\times Y^\dagger_\beta \subset T^*\mathbb{A}^n$.  Since $N^\dagger$ is a coisotropic linear submanifold of a symplectic vector space, it has a linear symplectic quotient
\begin{equation}\label{coisotropic projection}
\Pi: N^\dagger\rightarrow \mathcal{S},
\end{equation}
 i.e., the quotient by the null foliation, which equals the quotient by a linear subspace of $N^\dagger$.  It can also be realized as a symplectic vector subspace of $N^\dagger$.
 In particular, if $\omega_{\cS}$ denotes the induced (linear) symplectic form on $\cS$ and $\omega_r$ denotes the pullback of $\omega_{T^*\mathbb{A}^n}$ to $N^\dagger$, then 
 \begin{equation}\label{equality of pullback forms}
 \omega_r = \Pi^*\omega_{\cS}.
 \end{equation}
 \begin{prop}\label{symplectic reduction up to etale cover}
 Under the map $\phi: U^-\times U_D \rightarrow T^*U^-\times \cS$ defined by $\phi(u,p) = (u, m(u), \Pi\big(dq(p)\big)$, we have 
$ \phi^*\big(\omega_{T^*U^-} + \omega_{\cS}\big) = a^*\omega_{T^*W}.$
 \end{prop}
 \begin{proof} 
 This follows from Proposition \ref{form of symplectic form} via \eqref{equality of pullback forms}.
 \end{proof}
 \noindent
  We have that $U_D\cap \mu^{-1}(0)$ is a symplectic submanifold of $N\subset T^*W$:
\begin{prop}
The Hamiltonian reduction of $U^-\times U_D$ for the $U^-$-action is isomorphic to $m\inv(0) = U_D\cap\mu^{-1}(0)$.
\end{prop}
\begin{proof}
It suffices to show that $m\circ p_2$ is a moment map for the $U^-$-action where $p_2$ denotes the second projection.  This follows immediately from the fact that $\Phi$ is a 
$U^-$-equivariant symplectic map, combined with the fact that, writing $\wt{\mu}: T^*U^-\times T^*W\rightarrow (\mathfrak{n}^-)^*$ for the moment map (which just equals projection on $(\mathfrak{n}^-)^*$), we have $m = \wt{\mu}\circ\Phi$.
\end{proof}

\begin{prop}
The pullback of the symplectic form along the composite map $U_D\cap \mu^{-1}(0)\rightarrow N^\dagger \rightarrow \mathcal{S}$ equals the symplectic form on $U_D\cap \mu^{-1}(0) = m^{-1}(0) \subset U_D$. It follows that the maps $\phi$ and $m^{-1}(0) = U_D\cap \mu\inv(0) \rightarrow \mathcal{S}$ are \'etale.
\end{prop}
\begin{proof}
The pullbacks of the symplectic forms are symplectic. 
\end{proof}

\section{Deformation Quantization near a KN Stratum}\label{DQ}
Throughout Section \ref{DQ} we assume a connected reductive $G$ acts on the smooth variety $W$.  We assume that $W$ has trivialized canonical bundle and that $G$ acts on $K_W = W\times\C$ via the character $\gamma_G$, yielding a canonical quantum comoment map $\mu^{\on{can}}$.  We also assume that $T^*W$ comes equipped with a KN stratification, and that the line bundle
$\mathscr{L}$ is trivialized with its $G$-equivariant structure defined by a character $\chi: G\rightarrow \Gm$.
We fix a $\Gm$-action on $T^*W$ defining a filtration of $\D(W)$, if $W$ is a representation, or the operator filtration if $W$ is not.

\subsection{A Quantum Bimodule}\label{quantum bimodule}
We fix a one-parameter subgroup $\beta: \Gm\rightarrow G$ labelling a KN stratum and write $\gK = U^-\ltimes \Gm$ as in Section \ref{sec:first-sympl}.
Passing to the subgroup $\gK$ and choosing a slice $U_D$ near a point $z\in Z_\beta^{ss}$  as in Section \ref{symplectic geom}, we write 
\bd
\bigvar := U^-\times U_D.
\ed
  An equivariant
deformation quantization $(\C[\bigvar][\![\hbar]\!], \ast)$  is provided by the Moyal-Weyl product $\ast$.   Henceforth we write $\theo^{\hbar}(\bigvar)$ for this deformation quantization (with the Moyal-Weyl product); similarly, we write $\theo^{\hbar}(\cS)$ for $\C[\cS][\![\hbar]\!]$ with its Moyal-Weyl product.
Fix the canonical quantum comoment map $\mu^{U^-}$ and abusively use the same notation for the classical comoment map associated to the quantum map at
$\hbar=0$.  Let
\bd
M_{U^-} =\theo^{\hbar}(\bigvar) /\theo^{\hbar}(\bigvar)\mu^{U^-}(\mathfrak{n}^-);
\ed
this is a left $\theo^{\hbar}(\bigvar)$-module.  Define a linear map 
$a: \C[\cS] \rightarrow \on{End}_{\theo^{\hbar}(\bigvar)}(M_{U^-})$
by $a(f)(m) = m\ast (\pi_{\cS}\circ \phi)^*(f)$, where $\phi$ is the map from Proposition 
\ref{symplectic reduction up to etale cover} and $\pi_\cS: T^*U^-\times \cS\rightarrow \cS$ is projection.
\begin{prop}\label{reduction is Moyal}
\mbox{}
\begin{enumerate}
\item The map $a$ is well-defined: for each $f\in\C[\cS]$,  $a(f)$ is an $\theo^{\hbar}(\bigvar)$-module endomorphism of $M_{U^-}$.
\item  The $\hbar$-linear extension of the map $a$ defines an algebra homomorphism
\bd
\theo^{\hbar}(\cS) \longrightarrow\on{End}_{\theo^{\hbar}(\bigvar)}(M_{U^-}).
\ed  
\end{enumerate}
\end{prop}
\begin{proof}
Elements of $(\pi_{\cS}\circ\phi)^*\C[\cS]$ are $U^-$-invariant, hence normalize $\mu^{U^-}(\mathfrak{n}^-)$ by part (1) of Definition \ref{qcomo}; the map $a$ is thus well-defined.  By 
Proposition \ref{symplectic reduction up to etale cover}, pullback by $\pi_{\cS}\circ\phi$ intertwines symplectic forms, hence Poisson brackets, hence
Moyal-Weyl quantizations.  
\end{proof}

\subsection{Comparison of Canonical Comoment Maps}\label{comparison of ccms}
Keeping the notation of Section \ref{quantum bimodule}, compatibility of quantization with \'etale maps implies:
\begin{lemma}\label{etale base change}
The natural
 $\gK$-equivariant map
$a\inv: \theo^{\hbar}(T^*W^\circ) \rightarrow \theo^{\hbar}(\bigvar)$
intertwines the Moyal-Weyl products.  
\end{lemma}
Lemma \ref{etale base change} allows us to consider the composite homomorphism
\bd
\mathfrak{g}\xrightarrow{\mu^{\on{can}}} \theo^{\hbar}(T^*W^\circ)\xrightarrow{a\inv} \theo^{\hbar}(\bigvar).
\ed
\begin{prop}\label{action of Euler vector fields}\mbox{}
Let $\mathsf{eu}_{T^*W}(\beta)$, $\mathsf{eu}_{\cS}(\beta)$ denote the (canonically-shifted) Euler operators for the $\beta$-action on $T^*W$  and
$\cS$, respectively: that is, the images of the canonical generator of $\on{Lie}(\Gm)$ under the canonical quantum comoment maps associated to the 
$\beta$-actions
(cf. Section \ref{diff ops}).  Then, letting $1\in M_{U^-}$ denote the canonical generator, we have
\bd
\mathsf{eu}_{T^*W}(\beta)\ast 1 = 1\ast \mathsf{eu}_{\cS}(\beta) - \frac{\hbar}{2} \on{wt}_{\mathfrak{n}^-}(\beta),
\ed
where $\on{wt}_{\mathfrak{n}^-}(\beta)$ denotes the weight of $\beta$ on $\bigwedge\nolimits^{\on{top}}\mathfrak{n}^-$.  
\end{prop}
We have maps
\bd
T^*U^-\times \cS \leftarrow U^-\times U_D \rightarrow T^*W
\ed
which are \'etale, symplectic, and equivariant for the $\Gm$-actions via $\beta$.  It follows that the canonically shifted Euler operators for the $\beta$-action are identified via pullbacks.  Hence it suffices to prove the desired equality of actions on the bimodule
$\theo^{\hbar}(T^*U^-\times \cS)/\theo^{\hbar}(T^*U^-\times \cS)\mu(\mathfrak{n}^-)$ (since this pulls back to $M_{U^-}$).  Using the exponential map to identify $\mathfrak{n}^-$ with $U^-$ equivariantly for $\beta$, one calculates in coordinates.   We explain this in more detail in Section \ref{beta moment maps} below.

\subsection{Proof of Proposition \ref{action of Euler vector fields}}\label{beta moment maps}
Since $U^-\times U_D\xrightarrow{a} T^*W\xrightarrow{dq} T^*\mathbb{A}^n$ is \'etale, we obtain a classical moment map $\mu_{\beta, U^-\times U_D} = a^*\mu_{\beta, T^*W} = a^*dq^*\mu_{\beta, T^*\mathbb{A}^n}$ on $U^-\times U_D$  by pullback.
  
The discussion above---specifically, Proposition \ref{symplectic embedding prop}---guarantees that 
\begin{equation}\label{moment map splitting}
\mu_{\beta, U^-\times U_D}(u,p) = \mu_{\beta,\mathcal{S}}(\Pi(p)) + \mu_{\beta, T^*U^-}(u, m(p))
\end{equation}
 (where $\Pi$ is defined as in \eqref{coisotropic projection}).   Choose a basis $e_1,\dots, e_s$ for $\mathfrak{n}^-$ consisting of $\beta$-weight vectors with corresponding coordinate functions $e_i^*$ on $U^-$.
  The exponential map identifies $U^-$ as a $\Gm$-variety (for the $\beta$-action) with $\mathfrak{n}^-$.  Using this identification, 
 we may then rewrite $\mu_{\beta, T^*U^-} = -\sum_i \on{wt}(e_i) e_i^* e_i$, where $e_i$ is viewed as a function on 
 $(\mathfrak{n}^-)^*$ and hence on $T^*U^-$ via projection and $\on{wt}(e_i)$ denotes the $\beta$-weight on $e_i$.  Then, writing $m_i(p) = \langle m(p), e_i\rangle$, \eqref{moment map splitting} becomes
 \begin{equation}\label{better moment map splitting}
 \mu_{\beta, U^-\times U_D}(u,p) = \mu_{\beta,\mathcal{S}}(\Pi(p)) -\sum_i \on{wt}(e_i) e_i^* m_i(p).
 \end{equation}
 We want also to compute the Poisson bracket $\{e_i^*, m_i(p)\}$ on $U^-\times U_D$.  To do this, note that, since $m$ is the pullback of the $U^-$-moment map, $dm_i(p) = -  i_{\wt{e}_i}a^*\omega_{T^*W}$ where 
 $\wt{e}_i$ denotes the vector field on $U^-\times U_D$ generated by $e_i\in \mathfrak{n}^-$.  But this is the constant coordinate vector field $e_i$ on $U^-$.  Thus:
 \begin{equation}\label{important poisson bracket}
  \{e_i^*, m_i(p)\} = -dm_i(p)(X_{e_i^*}) = (a^*\omega_{T^*W})(\wt{e}_i, X_{e_i^*})  = -\wt{e}_i(e_i^*)  = -1.
 \end{equation}
 One then has:
\begin{lemma}\label{hbar shift}
Use the notation above.
 Under the Moyal-Weyl product on $U^-\times U_D$, we have 
 \begin{equation}\label{star product calc}
 e_i^*\ast m_i(p) = e_i^*\cdot m_i(p) - \frac{\hbar}{2}.
 \end{equation}
 \end{lemma}
 
 \begin{proof}
It follows from Formula \eqref{important poisson bracket} that $e_i^*\ast m_i(p) = e_i^*m_i(p) - \frac{\hbar}{2} + \mathcal{O}(\hbar^2)$.  However, note that, by Proposition \ref{form of symplectic form} and the $U^-$-invariance of the Poisson bivector field $P$ (thought of as a bidifferential operator), $P$ takes the form
\bd
P = \sum_i \left(\frac{\partial}{\partial e_i} \otimes Y_i + Y_i'\otimes \frac{\partial}{\partial e_i}\right) + \sum_k Z_k\otimes Z_k'
\ed
where each $Y_i, Y_i', Z_k, Z_k'$ is a vector field on $U^-\times U_D$ pulled back from $U_D$.  It follows that 
$P(e_i\otimes m_i(x)) = 1\otimes F$ for some function $F$ (which in fact is a scalar by \eqref{important poisson bracket}!) and thus
$P^n(e_i\otimes m_i(x)) = 0$ for $n\geq 2$; hence the definition of the Moyal product yields \eqref{star product calc}, as desired.
  \end{proof}

\begin{proof}[Proof of Proposition \ref{action of Euler vector fields}]
We first note that it follows from Formula \eqref{basic Euler op} that the Euler vector field $\mathsf{eu}(\beta)$ on a representation $\mathbb{A}^n$ is identified via the symmetrization map with the ``canonical'' classical moment map for the $\Gm$-action via $\beta$, that is, the moment map that takes the value $0$ at the base point.  Taking note of Remark \ref{compatibility of cans},  we thus obtain:
\begin{align*}
\on{Symm}\inv\big(\mathsf{eu}_{T^*W}(\beta)\big)\ast 1 & = a^*dq^*\mu_{\beta, T^*\mathbb{A}^n} =  \Pi^*\mu_{\beta,\mathcal{S}} - \sum_i \on{wt}(e_i) e_i^*\cdot m_i(p)
\hspace{1em} &  \text{by \eqref{better moment map splitting}}\\
& = 1\ast \on{Symm}\inv\left[\mathsf{eu}_{\cS}(\beta) - \sum_i\on{wt}(e_i) \left(e_i^*\ast m_i(p) +\frac{\hbar}{2}\right)\right] & \text{by Lemma \ref{hbar shift}}\\
& = 1\ast \on{Symm}\inv\big(\mathsf{eu}_{\cS}(\beta)\big) - \frac{\hbar}{2} \on{wt}_{\mathfrak{n}^-}(\beta),
\end{align*}
as claimed.
\end{proof}

\subsection{A Split Surjection of DQ Modules}\label{split surjection of DQ modules}
We maintain the notation of Sections \ref{quantum bimodule} and \ref{comparison of ccms}, 
we abusively write $\D(\cS)$ for the Weyl algebra associated to the linear symplectic space $\cS$.  
Suppose that $\mathsf{c}$ satisfies the condition of Corollary \ref{beta split surjection} for $\D(\cS)$, and let $M_{\mathsf{c}}^{\beta}(\chi^\ell)\otimes V_\ell\rightarrow M_{\mathsf{c}}^\beta$ denote the split surjection of $\D(\cS)$-modules from \eqref{evaluation}.  Note that by \cite[page~10]{McN1}, 
$M_{\mathsf{c}}^\beta(\chi^\ell)\cong M_{\mathsf{c}-\ell d\chi}\otimes \chi^{\ell}$.   

Recall that $\gK=U^-\ltimes \beta(\Gm)$.  Let $\mathsf{k} = \on{Lie}(\gK)$. 
By Section \ref{from D to DQ} and Proposition \ref{reduction is Moyal}, we have a homomorphism $\D(\cS)\rightarrow \End(M_{U^-})[\hbar\inv]$; tensoring up yields
 a split surjection of $\gK$-equivariant $\theo^{\hbar}(U^-\times U_D)$-modules
\begin{equation}\label{first induced sequence}
M_{U^-}[\hbar\inv]\otimes_{\D(\cS)}M_{\mathsf{c}}^{\beta}(\chi^\ell)\otimes V_\ell \rightarrow M_{U^-}[\hbar\inv]\otimes_{\D(\cS)}M_{\mathsf{c}}^{\beta}.
\end{equation}
\begin{prop}\label{heres the split surjection}
Suppose that $\mathsf{c}$ satisfies the condition of Corollary \ref{beta split surjection} for $\D(\cS)$.  
Then the split surjection \eqref{first induced sequence} is isomorphic as a map of $\gK$-equivariant DQ modules with a split surjection
  \begin{equation}\label{the honest split surjection}
\big(\theo^{\hbar}(\bigvar)/\theo^{\hbar}(\bigvar)\mu^{\on{can}}_{c-\ell d\chi}(\mathsf{k})\big)[\hbar\inv]\otimes \chi^{\ell}\otimes V_\ell
\longrightarrow 
\big(\theo^{\hbar}(\bigvar)/\theo^{\hbar}(\bigvar)\mu^{\on{can}}_{c}(\mathsf{k})\big)[\hbar\inv],
\end{equation}
 where $c:= \mathsf{c}+ \frac{1}{2}\on{wt}_{\mathfrak{n}^-}(\beta)$. 
 \end{prop}
 \begin{proof}
 This is immediate from Lemma \ref{momentary agreement} and Proposition \ref{action of Euler vector fields}.
 \end{proof}

\subsection{Vanishing for $\D$-Modules Along a KN Stratum}\label{reduction}
We now want to study $(G,c)$-equivariant $\D$-modules microlocally near a KN stratum using the tools described in the previous section.  
In particular, we want to establish a vanishing statement that we can use in inductive arguments.  

Thus,
suppose $M$ is a $(G,c)$-equivariant 
$\D_W$-module and that, with respect to the filtration that we are fixing, we have $SS(M)\subseteq S_{\geq\beta}$ (a union of KN strata).  The main aim of Section \ref{reduction} is to prove that, if an appropriate condition on $c$ is satisfied, then any $G$-invariant section of $M$ lies in a submodule of $M$ with singular support in $S_{>\beta}$.  We do this by first reducing $M$ to an equivariant Moyal-Weyl module on an open set of $\cS$ and then using Corollary \ref{beta split surjection} for the residual $\beta$-action.

Hence, let $M_c^G = \D/\D\mu^{\on{can}}_c(\mathfrak{g})$ (notation as in Section \ref{quantum notation}), and suppose we have a (weakly) $G$-equivariant map $\phi: M_c^G\rightarrow M$.

First, suppose that $W$ is a $G$-representation.  As before, we write the list of $\mathsf{T}$-weights on the $d$-dimensional vector space $W$ as $\{\alpha_1, \dots, \alpha_d\}$.  
Recall the definition of $I_{G,T^*W}(\beta)$ from Section \ref{parameter shifts}.  We will prove:
\begin{thm}\label{reductive inductive-rep}
Suppose $SS(M)\subseteq S_{\geq \beta}$ and that $c$ satisfies:
\begin{equation}\label{main condition}
c(\beta)  \notin \left(I_{G,T^*W}(\beta) + \on{wt}_{\mathfrak{n}^-}(\beta)  + \frac{1}{2}\sum_{i=1}^d |\alpha_i\bullet\beta|\right).
\end{equation}
Then $m = \phi(\mathbf{1})$ lies in a $G$-invariant submodule $M'$ of $M$ with singular support $SS(M')\subseteq S_{>\beta}$.
\end{thm}

We begin by fixing some notation and observing a formula.  We use the conventions of Section \ref{symplectic geom}.  Thus, with $\beta: \Gm\rightarrow G$ a fixed KN one-parameter subgroup, $U^- = U^-(\beta)$ denotes the unipotent subgroup associated to the Lie subalgebra $\mathfrak{n}^- = \mathfrak{n}^-(\beta)\subset \mathfrak{g}$ on which $\beta$ has negative weights.  Our construction of the slice $U_D$ to the infinitesimal $\mathfrak{n}^-$-action near a point of $Z_\beta$ and of the reduced space $\cS$ implies that, as $\Gm$-representations via $\beta$, we have 
\begin{equation}\label{splitting for reduction}
T^*W \cong \cS \oplus T^*\mathfrak{n}^-.
\end{equation}  N
ote that this implies the following formula.  For a $\Gm$-representation $\mathsf{V}$, let $\on{wt}(\mathsf{V})$ denote the list of weights {\em with multiplicities included}: so if the weight $3$ appears four times, then the list $\on{wt}(\mathsf{V})$ will include the number $3$ four times.  Recall that we write
$\on{wt}_{\mathfrak{n}^-}(\beta) = -\sum_{w_k\in \on{wt}(\mathfrak{n}^-)} |w_k|$ for the sum of $\beta$-weights on $\mathfrak{n}^-$ (which, by definition of $\mathfrak{n}^-$, are all negative).  
It follows from the direct sum decomposition \eqref{splitting for reduction} that the following holds for $\beta$-weights:
\begin{equation}\label{relationship of weight sums}
 2\sum_{i=1}^d |\alpha_i\bullet\beta| = \sum_{w_j\in\on{wt}(T^*W)} |w_j|  = \left(\sum_{w_j\in \on{wt}(\cS)} |w_j|\right) - 2\on{wt}_{\mathfrak{n}^-}(\beta).
 \end{equation}
 Thus, the condition \eqref{main condition} becomes 
\bd
c(\beta) \notin  I_{G,T^*W}(\beta) + \frac{1}{4}\left(\sum_{w_j\in \on{wt}(\cS)} |w_j|\right) + \frac{1}{2}\on{wt}_{\mathfrak{n}^-}(\beta).
\ed
The splitting \eqref{splitting for reduction} implies that $I_{G,T^*W}(\beta) = I_{\Gm, \mathcal{S}}(\beta)$.  Thus, \eqref{main condition} is equivalently written as 
\begin{equation}\label{the condition we will use}
c(\beta)  - \frac{1}{2}\on{wt}_{\mathfrak{n}^-}(\beta) \notin  I_{\Gm, \mathcal{S}}(\beta) + \frac{1}{4}\left(\sum_{w_j\in \on{wt}(\cS)} |w_j|\right).
\end{equation}
Consequently, Theorem \ref{reductive inductive-rep} is a special case of the following.  

Let $W$ be a smooth quasiprojective variety with trivialized canonical bundle and canonical quantum comoment map $\mu^{\on{can}}$.  Suppose $T^*W$ is equipped with trivial line bundle $\mathscr{L} = \theo_{T^*W}$ with $G$-equivariant structure twisted by the character $\chi: G\rightarrow \Gm$, and that $T^*W$ is equipped with a KN stratification.  
\begin{thm}\label{reductive inductive}
Suppose $SS(M)\subseteq S_{\geq \beta}$ and that $c$ satisfies \eqref{the condition we will use}.  If $\phi: \mathscr{M}_c\rightarrow M$ is a $G$-equivariant homomorphism,
then $m = \phi(\mathbf{1})$ lies in a $G$-invariant submodule $M'$ of $M$ with singular support $SS(M')\subseteq S_{>\beta}$.
\end{thm}

\begin{proof}[Proof of Theorem \ref{reductive inductive}]
We may replace $M$ by 
$\phi(\mathscr{M}_c)$ and hence assume that $M$ is cyclic, generated by the $G$-invariant vector $m = \phi(\mathbf{1})$.
We give $M$ the induced filtration from $\mathscr{M}_c$; hence we get surjections $\cR(\mathscr{M}_c)\twoheadrightarrow \cR(M)$ and 
\begin{equation}\label{Rees surjection}
\cR(\mathscr{M}_c)^{\hbar}\twoheadrightarrow \cR(M)^\hbar
\end{equation}
 of Rees modules and DQ modules.    Keep the notation of Section \ref{split surjection of DQ modules}.  It will suffice to show that there is an open cover of $T^*W\smallsetminus S_{>\beta}$ by affines $U_\alpha$ such that the map \eqref{Rees surjection} (equivalently, its target) vanishes on restriction to each $U_\alpha$, or in other words that $\cR(M)^\hbar|_{U_\alpha}=0$.  
 
 In fact, however, it suffices to do something weaker.  By Proposition \ref{induction to DQ}, $\cR(M)^\hbar$ is supported on $S_{\geq\beta}$. Since $\cR(M)^\hbar$ is weakly $G$-equivariant, $U^-\cdot Y_\beta$ is open in $S_\beta$, and $S_{\beta} = G\cdot (U^-\cdot Y_{\beta})$ (where, as in Section \ref{symplectic geom}, $U^-$ is the negative unipotent subgroup associated to $\beta$), it suffices to show that each restriction $\cR(M)^\hbar|_{U_\alpha}$ is zero for some collection $\{U_\alpha\}$ of open sets in $T^*W\smallsetminus S_{>\beta}$ whose union contains
 $U^-\cdot Y_\beta$.   By Proposition \ref{existence of U_D}(ii), the images $a(U^-\times U_D)$ of the affine varieties $U^-\times U_D$ constructed in Section 
 \ref{symplectic geom}  form such a collection of $U_\alpha$.  
 Now, applying $a\inv$ to \eqref{Rees surjection}, inverting $\hbar$, and
 ``forgetting'' to $\gK = U^-\ltimes\beta(\Gm)$ gives a $\gK$-equivariant surjection
 \begin{equation}\label{DQ surjection}
 \big(\theo^{\hbar}_{\bigvar}/\theo^{\hbar}_{\bigvar}\mu^{\on{can}}_{\mathsf{c}}(\mathsf{k})\big)[\hbar\inv]
 \xrightarrow{-\cdot m} a\inv\cR(M)^{\hbar}[\hbar\inv],
 \end{equation}
 and by the above discussion it suffices to show that each map \eqref{DQ surjection} vanishes.

 In fact, it is better to replace $a\inv\cR(M)^{\hbar}[\hbar\inv]$ by a slightly smaller module.  Namely, note that, since $S_\beta$ is smooth and $a$ is \'etale,
 $a\inv(S_\beta)$ is a finite disjoint union $a\inv(S_\beta) = \coprod C_i$ of closed, $U^-$-invariant subvarieties $C_i$, one connected component $C_0$ of which is $U^-\times (Y_\beta\cap U_D)\subseteq
 U^-\times U_D$.  The decomposition of $a\inv(S_\beta)$ into connected components determines a direct sum decomposition 
$a\inv\cR(M)^{\hbar}[\hbar\inv] = \oplus M_i := \oplus a\inv\cR(M)^{\hbar}[\hbar\inv]|_{C_i}$, with the property that the projection of the multiplication-by-$m$ map to any single summand $M_i$ is zero if and only if $\on{mult}_m$ was zero.   Let $M'$ denote the direct summand corresponding to the component 
$C_0 = U^-\times (Y_\beta\cap U_D)$.

Thus, we have a $\gK$-equivariant surjective map
\bd
\big(\theo^{\hbar}(\bigvar)/\theo^{\hbar}(\bigvar)\mu^{\on{can}}_{\mathsf{c}}(\mathsf{k})\big)[\hbar\inv]
 \xrightarrow{-\cdot m} M',
 \ed
 with $\on{supp}(M')\subset U^-\times Y_\beta$.  Let 
$M'(0) = \theo^{\hbar}(\bigvar)\cdot m\subset M';$
 this is a lattice in $M'$.

Assume the condition \eqref{the condition we will use} is satisfied.  
Then, writing
$\mathsf{c} = c(\beta) - \frac{1}{2}\on{wt}_{\mathfrak{n}^-}(\beta),$
Proposition \ref{heres the split surjection} gives a split surjection 
 as in \eqref{the honest split surjection}.

\begin{claim}\label{nonzeroness}
If $M'\neq\{0\}$, then $M'(0)/\hbar M'(0)$ has nonzero $(\chi\circ\beta)^\ell$-isotypic component for all $\ell\ll 0$: that is, 
\bd
\Hom_{\Gm}\big((\chi\circ\beta)^\ell, M'(0)/\hbar M'(0)\big) = \big[M'(0)/\hbar M'(0)\otimes (\chi\circ\beta)^{-\ell}\big]^{\Gm} \neq 0.
\ed
\end{claim}
\begin{proof}[Proof of Claim]
If $M'\neq \{0\}$, then the split surjection of \eqref{the honest split surjection} shows that 
\bd
\Hom(\big(\theo^{\hbar}(\bigvar)/\theo^{\hbar}(\bigvar)\mu^{\on{can}}_{c-\ell d\chi}(\mathsf{k})\big)[\hbar\inv]\otimes \chi^{\ell}, M')^{\gK}\neq 0
\ed
for $\ell \ll 0$.
In particular, $\Hom_{\Gm}(\big((\chi\circ\beta)^\ell, M'\big)\neq 0$.  If $n\in M'$ is a nonzero $(\chi\circ\beta)^{\ell}$-isotypic vector, then there is some 
$a$ for which $\hbar^a \cdot n \in M'(0)\smallsetminus \hbar M'(0)$, and the image of $\hbar^a \cdot n$ in $M'(0)/\hbar M'(0)$ is thus a nonzero $(\chi\circ\beta)^{\ell}$-isotypic vector.
\end{proof}
\begin{claim}\label{vanishing on graded}
$\Hom_{\Gm}\big(\chi^\ell, M'(0)/\hbar M'(0)\big) = 0$ for $\ell\ll 0$. 
\end{claim}
\begin{proof}[Proof of Claim]
The quotient $M'(0)/\hbar M'(0)$  is a finitely generated $\C[\bigvar]$-module set-theoretically supported on $U^-\times (Y_\beta\cap U_D)$.
By Corollary \ref{restriction of semi-invariants}, this support is cut out by the $\chi^q$-semi-invariants in $\C[\bigvar]$ for $q\gg 0$.  Thus, for $q\gg 0$, every $\chi^q$-semi-invariant  in $\C[\bigvar]$
 kills $M'(0)/\hbar M'(0)$.  It follows by
\eqref{semi-invariant vs isotypic} that for $q\gg 0$, every $\chi^{-q}$-isotypic vector $f\in \C[\bigvar]$ kills $M'(0)/\hbar M'(0)$.  But now $M'(0)/\hbar M'(0)$ is generated by a $\beta(\Gm)$-invariant vector, the image of $\mathbf{1}$.  Thus, if $n\in M'(0)/\hbar M'(0)$, $n=f\cdot\mathbf{1}$ for some $f\in\C[\bigvar]$; and if $n$ is $\chi^\ell$-isotypic, we may choose a 
$\chi^\ell$-isotypic $f \in \C[\bigvar]$ in this expression (using reductivity of $\Gm$).  Now for $q = -\ell\gg 0$, we use that every such $f$
kills $M'(0)/\hbar M'(0)$ to conclude that $n = f\mathbf{1} = 0$.  This proves the claim. 
\end{proof}
\noindent
Claims \ref{nonzeroness} and \ref{vanishing on graded} imply that $M'=0$; by the discussion above, this suffices.
\end{proof}

\section{$t$-Exactness for Quantum Direct Images and Microlocalization}\label{main theorems section}
In this section we first prove the main vanishing statement for quantum Hamiltonian reduction and the resulting existence of a certain split surjective homomorphism, treating first the case in which the affine variety $W$ is a $G$-representation (Theorem \ref{reductive strong vanishing}, Section \ref{Vanishing and Split Surjections}) and then its extension to more general smooth quasiprojective $W$ (Theorem \ref{general affine strong vanishing}, Section \ref{Vanishing and Split Surjections for Affine}).  

At the time of writing, there are several different technical frameworks available for quantum geometry.  Although one essentially knows that all of these are equivalent, there is not yet a systematic treatment of such equivalences between all frameworks available in the literature.  Hence, in Sections \ref{W framework} and \ref{McN framework} we 
briefly explain how Theorem \ref{reductive strong vanishing} implies $t$-exactness results in two such frameworks, the $\cW$-algebras of Kashiwara, Schapira, et al. and the categorical framework of \cite{McN1}.  The generalization to other frameworks is equally straightforward.  

It is convenient to use the following.
\begin{lemma}\label{KN coisotropic lemma}
Suppose that $M$ is a $(G,c)$-equivariant $\D$-module with $SS(M)\subseteq S_{\geq\beta}$.  If $SS(M)\cap S_\beta \neq\emptyset$, then $S_\beta\cap\mu\inv(0)$ contains a nonempty coisotropic subvariety.
\end{lemma}
\begin{proof}
Since $M$ is $(G,c)$-equivariant, $SS(M)\subseteq \mu\inv(0)$.  Also, $SS(M)$ is coisotropic; in particular, each irreducible component of $SS(M)$ is coisotropic.  Since $S_\beta$ is open in $S_{\geq\beta}$, it follows from the hypothesis of the lemma that  $\big(S_\beta\cap\mu\inv(0)\big)\cap SS(M) = S_\beta\cap SS(M)$ is empty or coisotropic.
\end{proof}
Recall from the introduction the subset $\mathsf{KN}$ of KN 1-parameter subgroups:
\bd
\mathsf{KN} = \big\{ \beta \; | \; S_\beta\cap \mu\inv(0) \; \text{contains a nonempty coisotropic subset}\big\}.
\ed

\subsection{Vanishing and Split Surjections for Representations}\label{Vanishing and Split Surjections}
Suppose first that $W$ is a representation of $G$.  Choose any refinement of the partial ordering on KN 1-parameter subgroups to a total ordering.

An ascending induction on $\beta$ (meaning a descending induction on strata!) yields:
\begin{thm}\label{reductive strong vanishing}
Let $W$ be a representation of $G$.  
Assume that, for every KN 1-parameter subgroup $\beta\in\mathsf{KN}$ for the $G$-action on $T^*W$,  $c$ satisfies 
\begin{equation}\label{main formula}
c(\beta)  \notin \left(I_{G,T^*W}(\beta) + \on{wt}_{\mathfrak{n}^-}(\beta) + \frac{1}{2}\sum_{i=1}^d |\alpha_i\bullet\beta|\right).
\end{equation}
  Then:
\begin{enumerate}
\item If $M$ is any object of $(\D,G,c)-\on{mod}$ with $SS(M)\subseteq (T^*W)^{uns}$, then $\Hom_{(\D,G,c)}(M_c, M)=0$.  
\item
For every $\ell\ll 0$, there is a finite-dimensional vector subspace 
\bd
V_\ell\subset \Hom_{(\D,G,c)}(M_c(\chi^\ell), M_c)
\ed
 for which the natural composite
evaluation map
\begin{equation}\label{G evaluation}
M_c(\chi^\ell)\otimes V_\ell \longrightarrow M_c(\chi^\ell)\otimes \Hom_{(\D,G,c)}\big(M_c(\chi^\ell), M_c\big) \longrightarrow 
M_c
\end{equation}
is a split surjective homomorphism of objects of $(\D,G,c)-\on{mod}$. 
\end{enumerate}
\end{thm}
\begin{proof}
(1) follows by induction on $\beta$ using Theorem \ref{reductive inductive-rep}.  For (2), given $\ell \ll 0$, the cokernel of the evaluation map $M_c(\chi^\ell)\otimes \Hom_{(\D,G,c)}\big(M_c(\chi^\ell), M_c\big) \longrightarrow 
M_c$ is a $(\D, G,c)$-module with $\chi$-unstable support.  Choose any finite-dimensional subspace $V_\ell$ as in the statement of the theorem 
such that the cokernel of the composite map \eqref{G evaluation} still has $\chi$-unstable support.  Part (1) of the theorem then yields a surjection
\bd
\Hom(M_c, M_c(\chi^\ell)\otimes V_\ell)\twoheadrightarrow 
\Hom(M_c, M_c),
\ed
and any element in $\Hom(M_c, M_c(\chi^\ell)\otimes V_\ell)$ in the preimage of $\on{Id}\in \Hom(M_c, M_c)$ provides a splitting as claimed in part (2).
\end{proof}
\subsection{Vanishing for Quasiprojective Varieties}\label{Vanishing and Split Surjections for Affine}
Now suppose that $W$ is a smooth, connected quasiprojective $G$-variety.  
We fix the operator filtration on $\D_W$.\footnote{We remark, however, that the statements and proofs would work more generally; for example, if $W$ is affine and equipped with a contracting $\Gm$-action commuting with $G$, we can give $\D(W)$ a {\em Kazhdan-type filtration} as in Section 4 of \cite{GG}.}    

We assume:
\begin{enumerate}
\item[(i)] the canonical bundle $K_W$ is trivial and is $G$-equivariantly isomorphic to $\theo_W$ twisted by a character $\gamma_G: G\rightarrow \Gm$.
\end{enumerate}  
We note that this assumption is harmless.  Indeed, replacing $W$ by the principal $\Gm$-bundle $\wt{W}\xrightarrow{\pi} W$ whose points correspond to nonzero vectors in $K_W$, we find that $\wt{W}$ is a smooth quasiprojective $G\times\Gm$-space for which $\pi^*K_W$ is a trivial, and then $K_{\wt{W}}$ is also trivial and $G\times\Gm$-equivariant.  Moreover, $K_{\wt{W}}$ extends  $G\times\Gm$-equivariantly to any smooth $G\times\Gm$-equivariant compactification of $\wt{W}$, and if $s$ is a nonvanishing section of $K_{\wt{W}}$,  $\on{div}(g\cdot s) = \on{div}(s)$ for all $g\in G$ (recall that $G$ is connected).  It follows that $G\times\Gm$ acts on the line $\C\cdot s$ by a character.  Now any Hamiltonian reduction of $W$ for $G$ comes from a Hamiltonian reduction of $\wt{W}$ for $G\times\Gm$.

We define the $\rho$-shift and canonical quantum comoment map as in Section \ref{diff ops}; we use these to define $(G,c)$-equivariant $\D$-modules.  We assume also that:
\begin{enumerate}
\item[(ii)] $T^*W$ comes with a KN stratification, and
\item[(iii)] the polarization $\mathscr{L}$ appearing in the definition of KN stratification is trivial, with $G$-action given by twisting the standard one on $\theo_{T^*W}$ by a character $\chi: G\rightarrow \Gm$.
\end{enumerate}  

\begin{remark}
By Proposition \ref{prop:KN-affine-var},
if $W$ is affine and the polarization $\mathscr{L}$ is trivial and determined by a character $\chi$, then $T^*W$ possesses a KN stratification.  
\end{remark}

For each connected component $Z_{\beta, i}$ of the $\beta$-fixed locus we define $I_{G,T^*W}(\beta, i)$ and $\on{abs-wt}_{N_{Z_{\beta,i}/T^*W}}(\beta)$ as in the introduction.

\begin{thm}\label{general affine strong vanishing}
Let $W$ be a smooth quasiprojective variety satisfying (i),(ii), and (iii) above.  Suppose that for every $\beta\in\mathsf{KN}$ and every connected component $Z_{\beta,i}$ of the $\beta$-fixed locus, we have 
\begin{equation}\label{second main formula}
c(\beta)  \notin \left(I_{G,T^*W}(\beta, i) + \on{wt}_{\mathfrak{n}^-}(\beta) + \frac{1}{4}\on{abs-wt}_{N_{Z_{\beta,i}/T^*W}}(\beta)\right).
\end{equation}
Then:
\begin{enumerate}
\item If $M$ is any object of $(\D_W,G,c)-\on{mod}$ with $SS(M)\subseteq (T^*W)^{uns}$, then $\Hom_{(\D,G,c)}(\mathscr{M}_c, M)=0$. 
\end{enumerate}
Suppose in addition that $W$ is an affine variety.  Then:
\begin{enumerate} 
\item[(2)]
For every $\ell\ll 0$, there is a finite-dimensional vector subspace 
\bd
V_\ell\subset \Hom_{(\D,G,c)}(M_c(\chi^\ell), M_c)
\ed
 for which the natural composite
evaluation map
\bd
M_c(\chi^\ell)\otimes V_\ell \longrightarrow M_c(\chi^\ell)\otimes \Hom_{(\D,G,c)}\big(M_c(\chi^\ell), M_c\big) \longrightarrow 
M_c
\ed
is a split surjective homomorphism of objects of $(\D,G,c)-\on{mod}$. 
\end{enumerate}
\end{thm}
\begin{proof}
Repeat the proof of \ref{reductive strong vanishing} using Theorem \ref{reductive inductive} in place of 
Theorem \ref{reductive inductive-rep}.  
\end{proof}

\subsection{Application to Microlocalization for $\cW$-Algebras}\label{W framework}
The split surjection \eqref{G evaluation} provides a versatile tool: applying any additive functor to it, we obtain a split surjection of the resulting objects. Thus, for example, inducing to modules over a $\Gm$-equivariant formal deformation quantization as in \cite[Chapter 6]{KSDQ} we also obtain split surjections of 
$\Gm$-equivariant DQ-modules; further applying symplectic reduction as in \cite[\S 2.5]{KR} then yields split surjections of quantized line bundles on the symplectic quotient corresponding to the characters $\chi^\ell$.   

More precisely, starting from the canonical $\cW$-algebra on $T^*W$ for a smooth affine variety $W$ with action of a reductive $G$ for which the classical moment map $\mu$ is flat, suppose one gets as the GIT quotient at the character $\chi$ a smooth symplectic variety $\resol = \mu\inv(0)/\!\!/_{\chi} G$.  
Let $\cW_{\resol}(c)$ be the $\cW$-algebra on $\resol$ obtained by quantum Hamiltonian reduction using the quantum comoment map $\mu^{\on{can}}_c$, as in \cite[\S 2.5]{KR}.  It is standard in GIT that the sequence of line bundles on $\resol$ associated to the characters $\chi^{\ell}$ is ample.  Hence:
\begin{thm}\label{W thm}
Suppose the hypothesis on $c$ of Theorem \ref{reductive strong vanishing} is satisfied.  Then for every good $\cW_{\resol}(c)$-module $M$, we have
$H^i(M)=0$ for $i\neq 0$.  In particular, the global section functor is an exact functor of good $\cW_{\resol}(c)$-modules.
\end{thm}
\begin{proof}
Inducing the split surjection of Theorem \ref{reductive strong vanishing} to $\cW_{\resol}(c)$-modules implies that condition (2.5) of Theorem 2.9 of \cite{KR} is satisfied.
\end{proof}

\begin{remark}
The argument of \cite{KR} using the split surjection is in essence extremely general, and can be readily adapted to essentially any other reasonable framework for sheaves of quantum algebras deforming a symplectic variety.
\end{remark}

\subsection{Application to Microlocalization for Localized Categories}\label{McN framework}
Assume that the classical moment map $\mu$ is flat.  We write
\bd
\cE_{\resol}(c)-\on{mod} \overset{\on{def}}{=} (\D(W),G,c)-\on{mod}/(\D(W),G,c)-\on{mod}^{uns}
\ed 
for the quotient of the category of $(G,c)$-equivariant $\D$-modules by the full subcategory of modules with unstable singular support; we let 
$D(\cE_{\resol}(c))$ denote its unbounded derived category.   In \cite{McN1} we define a functor
$D(\cE_{\resol}(c))\xrightarrow{\mathbb{R}\map_*} D(U_c)$
from the microlocal derived category to the derived category of left modules for the algebra $U_c = M_c^{G}$.   Note that the microlocal derived 
category depends on the choice of group character $\chi: G\rightarrow \Gm$.   Corollary \ref{reductive strong vanishing} yields:

\begin{thm}\label{t-exactness}
Assume that $c$ satisfies the hypothesis of Theorem \ref{reductive strong vanishing}.
Then the functor $\mathbb{R}\map_*$ is $t$-exact.  
\end{thm}
As an alternative to the use of Theorem \ref{reductive strong vanishing}(2), we give a
proof based on Theorem \ref{reductive strong vanishing}(1).
\begin{proof}
Recall the following from \cite{McN1}.
In the notation of \cite{McN1}, $\pi_c: (\D,G,c)-\on{mod}\rightarrow \cE_{\resol}(c)-\on{mod}$ denotes the projection on the quotient category and 
 $\Gamma_c: \cE_{\resol}(c)-\on{mod}\rightarrow (\D,G,c)-\on{mod}$ is the right adjoint to the projection $\pi_c$.  Then, for any object $M$ of 
 $\cE_{\resol}(c)-\on{mod}$, we have 
 \bd
\map_*\big(\pi_c(M)\big) = \Hom_{\D}\big(M_c, \Gamma_c\circ \pi_c(M)\big)^{G}.
 \ed
 Suppose
$ 0\rightarrow M_1\rightarrow M_2\rightarrow M_3\rightarrow 0$
 is exact in $\cE_{\resol}(c)-\on{mod}$.  Then 
 \bd
 0\rightarrow \Gamma_c(M_1)\rightarrow \Gamma_c(M_2)\rightarrow \Gamma_c(M_3)
 \ed
 is exact in $(\D,G,c)-\on{mod}$, and, furthermore, 
 \bd
 SS\left[\on{coker}\big(\Gamma_c(M_2)\rightarrow \Gamma_c(M_3)\big)\right] \subseteq T^*W^{uns}
 \ed
 (this last property is standard but it is immediate, for example, from Theorem 5.8 of \cite{McN1}). 
 Theorem \ref{reductive strong vanishing} thus implies that 
$ \Hom_{\D}\big(M_c, \on{coker}\big(\Gamma_c(M_2)\rightarrow \Gamma_c(M_3)\big)\big)^{G} = 0.$
 Since $\Hom_{\D}(M_c, -)^{G}$ is an exact functor of $(G,c)$-equivariant $\D$-modules (cf. Lemma 3.4 of \cite{McN1}), it follows that
$ 0\rightarrow\map_*M_1\rightarrow \map_*M_2\rightarrow \map_*M_3\rightarrow 0$
 is exact in $U_c-\on{mod}$.  The theorem follows. 
 \end{proof}

Theorem 1.1 of \cite{McN1} states that
the left adjoint $\mathbb{L}\map^*$ of $\mathbb{R}\map_*$ is cohomologically bounded if and only if $\mathbb{R}\map_*$ is an equivalence of derived categories.  In particular, combining Theorem 1.1 of \cite{McN1} and Theorem \ref{t-exactness} above, we find:
\begin{corollary}\label{abelian equiv}
Suppose that 
\begin{enumerate}
\item $\mathbb{L}\map^*$ is cohomologically bounded.  
\item The Lie algebra character $c$ satisfies the hypothesis of Theorem \ref{reductive strong vanishing}.
\end{enumerate}
Then $\map^*, \map_*$ form mutually quasi-inverse equivalences of abelian categories.
\end{corollary}

\section{Example: type $A$ Spherical Rational Cherednik Algebra}\label{type A spherical section}
In this section we give a quick and easy deduction of a slightly weaker form of the exactness part of \cite{KR}.

Fix $n\geq 1$.  
Let $W = \mathfrak{gl}_n \times \C^n$ with $G = GL_n$ acting in the usual way.  Let $\sT$ denote the maximal torus of diagonal matrices in $G$, with Lie algebra $\mathfrak{t}$ and standard rational inner product.    If $e_i$ denotes the $i$th standard basis vector in
$\C^n = \mathfrak{t}\cong \mathfrak{t}^*$, the weights of $\sT$ on $W$ are $e_i - e_j$, $1\leq i, j\leq n$, and $e_i$, $1\leq i\leq n$.  
Fix the determinant character $\mathsf{det}: G\rightarrow \Gm$.  Then $\lambda = d\mathsf{det}|_{\sT} = \sum_{i=1}^n e_i$.  Thus 
$\lambda\bullet (e_i-e_j) = 0$ for all $i,j$ and $\lambda\bullet e_i = 1$ for all $i$. 

Choose a subset $\boldsymbol{\alpha} = \{\alpha_i\}$ of the weights of $\sT$ on $W$ to produce a KN 1-parameter subgroup.  Via the action of the Weyl group $W=S_n$, we may assume that the subset of the weights $e_1,\dots, e_n$ that appear in $\boldsymbol{\alpha}$ is exactly $e_{k+1}, \dots, e_n$ for some $0\leq k\leq n-1$, or is empty.  Suppose that the subset of such weights appearing is nonempty, and let $e_i-e_j$ be an additional weight in $\boldsymbol{\alpha}$.  If $i,j\geq k+1$ then $e_i-e_j$ lies in the span of $e_{k+1}, \dots, e_n$.  If $i,j\leq k$ then $e_i-e_j$ is orthogonal to the span of $e_{k+1}, \dots, e_n$.  If exactly one of $i,j$---say, $i$---lies in $1, \dots, k$ then the span of $e_i-e_j, e_{k+1}, \dots, e_n$ equals the span of $e_j, e_{k+1}, \dots, e_n$.  Thus, we may assume, when computing the span of the elements of $\boldsymbol{\alpha}$, that $\boldsymbol{\alpha}$ consists of the weights $e_{k+1}, \dots, e_n$ together with some subset of the weights $e_i-e_j$ with $1\leq i,j\leq k$.  
The projection of $\lambda$ on the orthocomplement to the 
span of weights is thus either $0$ (if $k=0$)  or  $\sum_{i=1}^k e_i$.  The first of these corresponds to the trivial 1-parameter subgroup and may be discarded.  Write $\beta_k$ for the 1-parameter subgroup corresponding to $\sum_{i=1}^k e_i$.  

We now compute the terms in Formula \eqref{main formula} for $\beta_k$.   Since $\beta_k \bullet \alpha_i$ is $\pm 1$ or $0$ for every $i$, we find
that $I(\beta_k) = \mathbb{Z}_{\geq 0}$.    The shift that appears in Formula \eqref{main formula} is 
\bd
\frac{1}{2}\sum_i |\alpha_i\bullet\beta_k| + \sum_{\gamma_i\in \on{wt}(\mathfrak{n}_{\beta_k}^-)} \gamma_i \bullet \beta.
\ed
Note that $|(e_i-e_j)\bullet\beta_k| = 1$ if and only if exactly one of $i, j$ lies in $\{1,2, \dots, k\}$ and is zero otherwise.  We thus get for the shift
\bd 
\frac{1}{2}\sum_{i\in\{1,\dots,k\}, j\in\{k+1,\dots, n\}} 2|(e_i-e_j)\bullet \beta| + \frac{1}{2}\sum_{i=1}^n e_i\bullet\beta + 
 \sum_{\gamma_i\in \on{wt}(\mathfrak{n}_{\beta_k}^-)} \gamma_i \bullet \beta
 \ed
 \bd
 = \frac{1}{2}(2(n-k)k) + \frac{1}{2} k + -(n-k)k =  \frac{k}{2}.
 \ed
 Consider the character $c\sum_{i=1}^n e_i$ on $\mathfrak{gl}_n$.  Note that, for the space $W$ above, 
 $\ds\rho = \frac{-1}{2} \sum_{i=1}^n e_i$ (since $\mathfrak{gl}_n$ is reductive, its weights sum to zero).   Write $c'\sum_{i=1}^n e_i = c\sum_{i=1}^n e_i - \rho$.  Theorem \ref{t-exactness} says that $t$-exactness holds provided that for $1\leq k \leq n$,
 \bd
 \left(c'\sum_{i=1}^n e_i\right) \bullet \beta_k \notin I(\beta_k) + \frac{k}{2} = \mathbb{Z}_{\geq 0} + \frac{k}{2}.
 \ed
 Since $\sum_{i=1}^n e_i \bullet \beta_k  = k$, this becomes
$\ds c'k \notin \mathbb{Z}_{\geq 0} + \frac{k}{2}$
 or $\ds c'\notin \frac{1}{k}\mathbb{Z}_{\geq 0} + \frac{1}{2}$.  Since $\ds c = c'-\frac{1}{2}$, we conclude that $t$-exactness holds provided 
$\ds c\notin \bigcup_{k=1}^n \frac{1}{k}\mathbb{Z}_{\geq 0}.$
Under the conventions of \cite{GGS}, our $c$ corresponds to their $-c$; hence in the notation of \cite{GGS} we have shown that
$t$-exactness holds provided $c$ is not a rational number of the form $\ds\frac{a}{b}$ for $a\leq 0$, $1\leq b\leq n$.  
 Now by \cite{GGS}, Theorem 2.8, we conclude:
 \begin{corollary}
 Exactness holds for microlocalization of the type $A$ spherical Cherednik algebra $eH_c e$ provided 
 \bd
 c\notin \Big\{-1 + \frac{a}{b} \;\; \Big| \;\; b\in \{1,2,\dots, n\}, a\in \mathbb{Z}_{\leq 0}\Big\}.
 \ed 
\end{corollary}

\bibliographystyle{alpha}

\end{document}